\documentclass[smallextended,numbook,runningheads]{svjour3}      
\usepackage{graphicx}
\usepackage{amsmath}
\usepackage{epstopdf} 
\usepackage{mathptmx}      
\journalname{Numer. Math.}
\usepackage{amsmath,ulem}
\usepackage{stmaryrd}
\usepackage{amssymb}
\usepackage{empheq}
\usepackage{cases}
\usepackage{cite}
\usepackage{caption,graphicx}

\usepackage{amsmath}
\usepackage{graphicx}
\usepackage{graphicx,epstopdf}
\usepackage{wrapfig}
\usepackage[caption=false]{subfig}
\usepackage{amssymb}
\usepackage{cases}
\usepackage{physics}
\usepackage{enumerate}
\usepackage{stmaryrd}
\usepackage{mathrsfs}
\usepackage{multirow}

\usepackage{algorithm}
\usepackage{algpseudocode}
\usepackage{indentfirst}

\usepackage{array} 

\usepackage{hyperref}
\usepackage{comment}
\hypersetup{backref=true,citebordercolor={1 1 1},citecolor=black}
\hypersetup{linkbordercolor={1 1 1},linkcolor=black,linktocpage=true}

\newcommand{\bR}{\mathbb{R}}

\newcommand{\be}{\begin{equation}}
\newcommand{\ee}{\end{equation}}
\newcommand{\ba}{\begin{array}}
\newcommand{\ea}{\end{array}}
\newcommand{\bea}{\begin{eqnarray}}
\newcommand{\eea}{\end{eqnarray}}
\newcommand{\beas}{\begin{eqnarray*}}
\newcommand{\eeas}{\end{eqnarray*}}

\newcommand{\n}{\boldsymbol{n}}
\newcommand{\X}{\boldsymbol{X}}
\newcommand{\q}{\boldsymbol{q}}
\newcommand{\w}{\boldsymbol{\omega}}
\newcommand{\dt}{\partial_t}
\newcommand{\Tau}{\boldsymbol{\tau}}
\newcommand{\ud}{\underline{D}}
\newcommand{\ds}{\nabla_{\Gamma}}
\newcommand{\lap}{\Delta_{\Gamma}}
\newcommand{\cur}{\mathcal{H}}

\smartqed  
\numberwithin{equation}{section}

\usepackage{newtxtext,newtxmath}

\begin{document}


\title{An energy-stable parametric finite element method for the Willmore flow in three dimensions}

\titlerunning{An energy-stable PFEM for Willmore flow}        

\author{Weizhu Bao \and
        Yifei Li \and Dongmin Wang
}

\authorrunning{W.~Bao, Y.~Li, D.~Wang } 

\institute{W.~Bao  \at
              Department of Mathematics, National
              University of Singapore, Singapore 119076\\
              Fax: +65-6779-5452, Tel.: +65-6516-2765,\\
              URL: https://blog.nus.edu.sg/matbwz/\\
              \email{matbaowz@nus.edu.sg}\\[1em]
           Y.~Li \at
              Mathematisches Institut, Universit\"at T\"{u}bingen, Auf der Morgenstelle 10., 72076 T\"{u}bingen, Germany\\
              \email{yifei.li@mnf.uni-tuebingen.de}\\[1em]
           D.~Wang \at
                Department of Mathematics, National University of Singapore, Singapore 119076\\
                School of Mathematical Sciences, Shanghai Jiao Tong University,   Shanghai 200240, China\\
   \email{e1583440@u.nus.edu}
              }

\date{Received: date / Accepted: date}

\maketitle

\begin{abstract}

This work develops novel energy-stable parametric finite element methods (ES-PFEM) for the Willmore flow and  curvature-dependent geometric gradient flows of surfaces in three dimensions. The key to achieving energy stability lies in the use of two novel geometric identities: (i) a reformulated variational form of the normal velocity field, and (ii) the incorporation of the temporal evolution of the mean curvature into the governing equations. These identities enable the derivation of a new variational formulation. By using the parametric finite element method, an implicit fully discrete scheme is subsequently developed, which maintains the energy dissipative property at the fully discrete level. Based on the ES-PFEM, comprehensive insights into the design of ES-PFEM for general curvature-dependent geometric gradient flows and a new understanding of mesh quality improvement in PFEM are provided. In particular, we develop the first {ES-PFEM} for the Gauss curvature flow of surfaces. Furthermore, a tangential motion control methodology is applied to improve mesh quality and enhance the robustness of the proposed numerical method. Extensive numerical experiments confirm that the proposed method preserves energy dissipation properties and maintains good mesh quality in the surface evolution under the Willmore flow.


\end{abstract}


\section{Introduction}
In this paper, we develop an energy-stable parametric finite element method (ES-PFEM) for the Willmore flow in three dimensions (3D) and extend it to related curvature-dependent geometric gradient flows. The Willmore energy is a central concept in geometric analysis proposed by Thomas Willmore \cite{willmore1993riemannian}, which quantifies the deviation of a surface from a sphere. For a closed surface $\Gamma$ embedded in 3D with the unit outward normal vector $\n$ and the mean curvature $\cur$, its Willmore energy is defined as 
\begin{equation}\label{Willmore-eng}
    W(\Gamma) \coloneqq \frac{1}{2}\int_{\Gamma} \cur^2\, \dd A.
\end{equation}
Among all closed surfaces, the sphere uniquely attains the minimal Willmore energy. Besides the sphere, other critical points of the Willmore energy, termed Willmore surfaces, are of fundamental importance in geometric analysis. The famous Willmore conjecture concerning the Willmore surfaces among tori was resolved by Marques and Neves \cite{marques2014min, marques2014willmore}, who proved that the Clifford torus achieves the minimal energy. For other applications in mathematics, we refer to \cite{Dall2012,francis1997minimax,kuwert2001willmore} and references therein.

Beyond its mathematical significance, the Willmore energy and the Willmore surfaces also exhibit profound physical applications across diverse domains. The Willmore energy was originally introduced by Leonhard Euler and Daniel Bernoulli to characterize nonlinear elastic energy in beam theory, enabling optimization of elastic material properties \cite{grinspun2003discrete, grinspun2005discrete, ciarlet2000mathematical}. In biological science, the Willmore functional is also known as the Helfrich energy, which successfully explains the biconcave shape of red blood cells and other membrane deformation phenomena \cite{helfrich1973elastic, deuling1976curvature, deuling1976red}. Additional applications span the black hole theory in astrophysics \cite{gibbons1983positive, hajicek1987origin}, denoising and restoration of surface in computer graphics \cite{bobenko2005discrete, bohle2008constrained, corman2024curvature}.

Geometric flow serves as a fundamental analytical tool for studying critical points of the surface energy functional in differential geometry. The Willmore flow, induced by the $L^2$-gradient flow of the Willmore energy functional, provides a natural approach for studying the Willmore surfaces. For an evolving closed surface $\Gamma(t)\subset \bR^3$ under the Willmore flow, its normal velocity $V$ is mathematically formulated (via thermodynamic variation of the Willmore 
energy (\ref{Willmore-eng})) as   
\begin{align} \label{eq: Willmore flow}
    V \coloneqq -\frac{\delta W(\Gamma)}{\delta \Gamma}=\lap \cur + |\mathbf{A}|^2 \cur - \frac{1}{2} \cur^3,
\end{align}
where $\lap = \nabla_{\Gamma}\cdot \nabla_{\Gamma}$ represents the Laplace-Beltrami operator, $\ds$ and $\nabla_{\Gamma}\cdot$ denote the surface gradient operator and the surface divergence operator, respectively, 
and $\mathbf{A} = \ds \n$ is the extended Weingarten map. The Willmore energy $W(\Gamma(t))$ satisfies the energy dissipative property:
\begin{equation}
    \frac{\dd W(\Gamma(t))}{\dd t} = -\int_{\Gamma(t)}V^2 \, \dd A\leq 0, \qquad t\ge0,
\end{equation}
with the dissipation rate given by the $L^2$-norm of the normal velocity $V$, confirming its interpretation as the $L^2$-gradient flow of the Willmore functional.

More generally, if we take the energy functional as $W(\Gamma) = \int_\Gamma f(\cur)\, \dd A$ for a general function $f(\cur)$, we can obtain the curvature-dependent geometric gradient flows via the $L^2$-gradient flow of the corresponding energy functional, with the normal velocity $V$ given as \cite{duan2021high}
\begin{equation}
    V \coloneqq \frac{\delta W(\Gamma)}{\delta \Gamma}= \lap (f'(\cur)) + |\mathbf{A}|^2f'(\cur) - \cur f(\cur).
\end{equation}
When $f(\cur) = \frac{\cur^2}{2}$, we recover the Willmore flow. Moreover, when $f(\cur) = 1$, the normal velocity becomes $V = -\cur$, which reduces to the mean curvature flow; and when $f(\cur)=\cur$, we obtain
$V=|\mathbf{A}|^2-\cur^2=-2\mathcal K$ with $\mathcal K=\det(\mathbf{A})$ being the Gaussian curvature which gives the Gauss curvature flow. {It corresponds to the classical case $\alpha=1$ in the wide family of Gauss curvature flow $V=-\mathcal K^\alpha$, up to a rescaling of time. These curvature-dependent geometric gradient flows, especially Willmore flow and Gauss curvature flow, have attracted considerable research interest; see \cite{andrews1999gauss,tso1985deforming,doerfler2019discrete} and references therein.

The Gauss curvature flow considered here is a classical fully nonlinear geometric flow that has been extensively studied in differential geometry \cite{tso1985deforming,andrews2000motion}. By the Gauss--Bonnet theorem, the change of the enclosed volume $M(t)$ during the evolution can be computed as follows:
\begin{equation}\label{volume rate of gauss flow}
\frac{\dd M}{\dd t} 
= \int_{\Gamma(t)} V \,\dd A 
= -2 \int_{\Gamma(t)} \mathcal{K} \,\dd A 
= -8\pi (1-g),
\end{equation}
where $g$ is the genus of the surface $\Gamma(t)$. For a smooth, strictly convex closed surface, the flow is well--posed and shrinks to a round point with constant volume decay rate $-8\pi$, confirming Firey's conjecture on the abrasion of worn stones \cite{firey1974shapes,andrews1999gauss}. Beyond its theoretical significance in geometry, the Gauss curvature flow also possesses physical and computational motivations with applications including mesh fairing and image denoising \cite{lee2005noise,brito2016image,zhao2006triangular}.}

Among these curvature-dependent geometric gradient flows, the Willmore flow \eqref{eq: Willmore flow} represents one of the most important and challenging examples. As a fourth-order highly nonlinear geometric partial differential equation, it poses considerable challenges to both numerical computation and theoretical analysis \cite{kuwert2001willmore, garcke2019willmore, rupp2023volume, kuwert2004removability}. Moreover, due to its gradient flow nature, it is desirable to develop a numerical method for the Willmore flow, which can preserve the energy dissipative property at the fully discrete level.

Different numerical methods have been developed for the numerical simulation of the Willmore flow, with comprehensive surveys available in \cite{bao2025energy}. In the following, we focus specifically on parametric-based numerical methods. The first parametric method for the Willmore flow of closed surfaces in $\mathbb{R}^3$ was proposed by Rusu \cite{rusu2005algorithm}, which was subsequently advanced by Dziuk with a rigorous energy-stable analysis for spatial semi-discretization \cite{dziuk2008computational}. Subsequently, Duan, Li and Zhang developed a fully discrete energy-stable scheme \cite{duan2021high} based on Dziuk's formulation by employing the averaged vector-field collocation method for temporal discretization. However, the above approaches lack tangential velocity, resulting in deteriorated mesh quality during evolution. To address this problem, Barrett, Garcke and N\"urnberg (denoted as BGN) proposed a parametric finite element method (PFEM) for mean curvature flow \cite{barrett2007variational} that incorporates tangential motion, thereby achieving favorable mesh quality. By combining with Dziuk's formulation, BGN further extended the PFEM for the Willmore flow in two dimensions (2D) and 3D under the semi-discretization in space, which achieves energy stability and good mesh quality \cite{Barrett2020}, while it is not valid for full discretization. Furthermore, the PFEM has been systematically extended to incorporate spontaneous curvature effects, area difference elasticity phenomena, and so on \cite{barrett2016computational, Barrett2020, barrett2017stable, barrett2021stable}. Nevertheless, rigorous energy stability analysis for parametric methods incorporating tangential motions for the Willmore flow in 3D, particularly fully discretized PFEMs, has yet to be established.

Recently, Bao and Li achieved significant theoretical progress by developing the first unconditionally energy-stable fully-discretized PFEM for the Willmore flow of planar curves \cite{bao2025energy}. Their key innovation involves utilizing a proper evolution equation for the mean curvature, which is inspired by the evolving surface finite element method established by Kov\'acs et al. \cite{kovacs2021convergent}. Subsequently, their methodology was extended by Ma et al. to the axisymmetric Willmore flow of closed surfaces \cite{ma2025energy}. Very recently, N\"urnberg, Garcke, and Zhao combined curvature evolution equations with the geometric framework of the BGN scheme through dual independent curvature discretizations \cite{garcke2025stable}, thereby developing an energy-stable parametric finite element method that preserves both tangential motion capabilities and rigorous energy dissipation properties. {However, the extension of these energy-stable methodologies for the Willmore flow of surfaces in 3D remains its difficulty}. Furthermore, for curvature-dependent geometric gradient flows, even for the simplest nonlinearity $f(\cur) = \cur$ in the Gauss curvature flow, the energy stability analysis remains limited to the curve case in 2D \cite{barrett2008parametric2}, where the Gaussian curvature $\mathcal{K}$ reduces to $\cur$.

There are two main difficulties in constructing an energy-stable PFEM (ES-PFEM) for curvature-dependent geometric gradient flows in 3D, which we illustrate using the Willmore flow as our primary example. First, the extended Weingarten map $\mathbf{A}$ introduces a fundamental difficulty. In 2D, $|\mathbf{A}|$ reduces to a scalar that coincides with the mean curvature $\kappa$, whereas this is not the case in 3D. The previous methods by Rusu and Dziuk avoid direct treatment of $\mathbf{A}$ through complicated geometric identities. Second, at the continuous level, the energy stability relies on the following Reynolds transport theorem and Huisken's evolution identity for the mean curvature \cite{huisken1984flow}.
\begin{equation}
    \frac{d}{dt} \int_{\Gamma(t)} \frac{1}{2} \cur^2\, \dd A = \int_{\Gamma(t)} \cur\, \partial_t \cur\, \dd A + \int_{\Gamma(t)} \frac{1}{2} \cur^2 \nabla_{\Gamma} \cdot (V\n) \, \dd A.
\end{equation}
Here we use $\partial_t$ to represent the material derivative in \cite{Barrett2020}. The first term on the right-hand side represents energy change and requires Huisken's identity $\partial_t \cur = - \lap V - |\mathbf{A}|^2 V$ \cite{huisken1984flow}. However, $\cur \lap V$ involves two consecutive spatial integrations by parts, demanding $V\in H^2$, which yields a mismatch of regularity with the linear finite elements used in PFEM. The second term represents domain change and contains a highly nonlinear Jacobian $\nabla_{\Gamma} \cdot (V\n)$, making it difficult to achieve energy stability at the fully discrete level.

The main objective of this paper is to develop the first unconditionally energy-stable fully discretized PFEM for the Willmore flow of closed surfaces, and then to extend it for general curvature-dependent geometric gradient flows including the Gauss curvature flow. Our framework introduces the following novelties:

     (i) We establish a novel evolution equation for the mean curvature $\cur$. It decomposes the consecutive spatial derivatives of Huisken's $\lap V$ through temporal derivatives, enabling finite element analysis with linear elements.
     
    (ii) We employ a modified version of the transport theorem
    \begin{equation}
        \frac{\dd}{\dd t} \int_{\Gamma(t)} f \,\dd A = \int_{\Gamma(t)} \dt f \, \dd A + \int_{\Gamma(t)} f \ds\X :\ds(\dt\X)\, \dd A,
    \end{equation}
    where $\X:=\X(t)$ is a representation of the surface $\Gamma(t)$.  The modified formulation of the domain change $\ds \X:\ds (\dt \X)$ can be analyzed through the same argument as the BGN scheme.

     (iii) We extend our ES-PFEM to the curvature-dependent geometric gradient flows of the energy functional $W(\Gamma) = \int_{\Gamma} f(\cur)\, \dd A$ with a weakly convex function $f$. Notably, this includes the Gauss curvature flow when $f(\cur) = \cur$, for which we provide the first energy-stable discretization by PFEM. When $f(\cur) \equiv 1$, the resulting geometric gradient flow is the mean curvature flow, and our ES-PFEM is reduced to the BGN scheme. Our method not only unifies the existing ES-PFEM for mean curvature flow, but also provides comprehensive insight for the design of ES-PFEM for general geometric gradient flows.

     (iv) We provide a new understanding of improving the mesh quality of the PFEM. The BGN scheme's good mesh quality { stems} from the mean curvature vector identity $\cur \n = -\lap \X $, which simultaneously governs the computation of mean curvature $\cur$. 
     Since our method employs an evolution equation for computing $\cur$, the mesh quality can be improved independently through an alternative tangential motion once a convective term about the tangential velocity is considered in this evolution equation.
     Based on this insight, we regulate discretization-induced tangential motion by coupling tangential velocity equations directly into the discrete system, introducing tunable hyperparameters that control tangential motion while preserving energy dissipation properties.

The remainder of this paper is organized as follows: In Section 2, by introducing two new geometric identities for the normal velocity vector and time evolution of the mean curvature, we derive a new variational formulation for the Willmore flow in 3D, and we show its energy dissipation through a new transport theorem. In Section 3, we propose an energy-stable full discretization of the new formulation by PFEM, and establish its unconditional energy dissipation. The ES-PFEM is extended for simulating the curvature-dependent geometric gradient flows in Section 4. In Section 5, we propose a tangential motion control methodology for the improvement of mesh quality in our ES-PFEM, and we also illustrate its Newton iterative solver. In Section 6, numerical examples are performed to test the convergence order,  numerical efficiency, energy dissipative property and meshing quality of our proposed ES-PFEM, and we compute lower bound of the Willmore energy by the numerical method.

\section{A new variational formulation for the Willmore flow in 3D} 
In this section, we present a new variational formulation for the Willmore flow base on two geometric identities for the normal velocity vector $V\n$ and the temporal evolution of the mean curvature $\cur$. The energy dissipation is established within the variational formulation by using a new transport theorem.   

\subsection{Basic notations and identities}
In this paper, we always assume that $\Gamma(t) \subset \bR^3$ is a sufficiently smooth, orientable and compact evolving surface with outward unit normal vector field $\n$, parameterized by a global parameterization $\X(\cdot,t)$ as
$$
\X(\cdot, t): \Gamma(0) \rightarrow \bR^3, \quad (\boldsymbol{\rho}, t)\mapsto \left(x_1(\boldsymbol{\rho},t),\ x_2(\boldsymbol{\rho},t),\ x_3(\boldsymbol{\rho},t) \right)^T,
$$
where $\Gamma(0)$ is the initial surface parameterized by the identity map $\X(\cdot, 0)$.

For a scalar function $f\in C^2(\Gamma)$ on surface $\Gamma$, let $\bar{f}$ be an differentiable extension in the neighborhood of surface $\Gamma$, the surface gradient operator on $\Gamma$ is defined via \cite{Barrett2020}
\begin{equation}
\ds f \coloneqq\left(I_3- \n\n^T\right) \nabla \bar{f} = \left(\ud_1 f, \ud_2 f,\ud_3 f\right)^T,
\end{equation}
where $\ud_i$ ($i=1,2,3$) represent the tangential derivative operators intrinsic to the surface $\Gamma$ \cite{deckelnick2005computation}, $I_3$ the $3\times 3$ identity matrix, and the surface Laplacian operator (also called Laplace-Beltrami operator) is given by
\be
\lap f \coloneqq \ds\cdot \ds f = \ud_1^2 f + \ud_2^2 f + \ud_3^2 f.
\ee
In the case of a vector valued function $\boldsymbol{f} = \left(f_1, f_2, f_3\right)^T\in [C^2(\Gamma)]^3$, we define the surface Jacobian, surface divergence, and surface Laplacian as follows:
\begin{subequations}
    \label{def surface differential operators}
    \begin{align}
        &\ds \boldsymbol{f} \coloneqq \left(\ds f_1, \ds f_2, \ds f_3 \right)^T = \begin{pmatrix}
            \ud_1 f_1 & \ud_2 f_1 & \ud_3 f_1 \\
            \ud_1 f_2 & \ud_2 f_2 & \ud_3 f_2 \\
            \ud_1 f_3 & \ud_2 f_3 & \ud_3 f_3
        \end{pmatrix}, \\
        &\ds \cdot \boldsymbol{f} \coloneqq \text{Tr}(\ds \boldsymbol{f}) = \ud_1 f_1 + \ud_2 f_2 + \ud_3 f_3, \\
        &\lap \boldsymbol{f} \coloneqq \left(\lap f_1, \lap f_2, \lap f_3\right)^T.
    \end{align}
\end{subequations}
For a matrix-valued function $\boldsymbol{F} = (F_{ij})_{i,j=1}^3 \in [C^1(\Gamma)]^{3\times 3} $, its surface divergence is defined as 
\begin{equation}\label{matrix surface divergence}
    \ds \cdot \boldsymbol{F}= \begin{pmatrix}
        \ud_1 F_{11} + \ud_2 F_{12} + \ud_3 F_{13} \\
        \ud_1 F_{21} + \ud_2 F_{22} + \ud_3 F_{23} \\
        \ud_1 F_{31} + \ud_2 F_{32} + \ud_3 F_{33}
    \end{pmatrix}.
\end{equation}

Here we gather the product rule for surface differential operators on $f\in C^1(\Gamma)$, 
$\boldsymbol{f}, \boldsymbol{g}\in [C^1(\Gamma)]^3$ and 
$\boldsymbol{F},\boldsymbol{G}\in [C^1(\Gamma)]^{3\times 3} $, the proof can be found in \cite[Lemma 7]{Barrett2020}:
\begin{subequations}
\begin{align}
    &\ds\cdot (f\boldsymbol{g}) = \boldsymbol{g}\cdot \ds f + f\ds \cdot \boldsymbol{g},\label{divergence multi}\\
    &\ds (\boldsymbol{f}\cdot \boldsymbol{g}) = (\ds\boldsymbol{f})^T \boldsymbol{g} + (\ds \boldsymbol{g})^T \boldsymbol{f},\label{gradient multi vector}\\
    &\ds (f\boldsymbol{g}) = \boldsymbol{g}(\ds f)^T + f\ds \boldsymbol{g},\label{gradient multi scalar}\\
    & \ds \cdot (\boldsymbol{f} \boldsymbol{g}^T) = (\ds \cdot \boldsymbol{g})\boldsymbol{f} + (\ds \boldsymbol{f}) \boldsymbol{g}, \label{divergence vector vector} \\
    & \ds \cdot (f \boldsymbol{G}) = f \ds \cdot \boldsymbol{G} + \boldsymbol{G} \ds f, \label{divergence scalar matrix}\\
    & \ds \cdot (\boldsymbol{F} \boldsymbol{g}) = (\ds \cdot \boldsymbol{F}^T) \cdot \boldsymbol{g} + \text{Tr}(\boldsymbol{F} \nabla_\Gamma \boldsymbol{g}). \label{divergence matrix vector}
\end{align}
\end{subequations}

We define the function space $L^2(\Gamma(t))$ by
\begin{equation}
    L^2(\Gamma(t)) \coloneqq \left\{f:\, \Gamma(t) \rightarrow \bR\, \Big|\, \int_{\Gamma(t)} f^2\, \dd A < +\infty \right\},
\end{equation}
equipped with the $L^2$ inner product $(\cdot, \cdot )_{\Gamma(t)}$ which is formulated by 
\begin{equation}
    (u, v)_{\Gamma(t)} \coloneqq \int_{\Gamma(t)} uv \, \dd A,\quad \forall u, v\,\in L^2(\Gamma(t)).
\end{equation}
The definition of $L^2(\Gamma(t))$ can easily be extended to vector-valued function space $[L^2(\Gamma(t)]^3$ and the matrix-valued function space $[L^2(\Gamma(t))]^{3\times 3}$. Specially, the inner product for two matrix-valued functions is emphasized by
\begin{equation}\label{inner product matrix}
    \langle\boldsymbol{U}, \boldsymbol{V}\rangle_{\Gamma(t)} \coloneqq \int_{\Gamma(t)} \boldsymbol{U}:\boldsymbol{V}\, \dd A, \quad \forall\, \boldsymbol{U}, \boldsymbol{V}\, \in [L^2(\Gamma(t))]^{3\times 3},
\end{equation}
where $\boldsymbol{U}:\boldsymbol{V} \coloneqq \text{Tr}\,(\boldsymbol{U}^T\boldsymbol{V})$ denotes the Hilbert-Schmidt inner product. And the Sobolev space $H^1(\Gamma(t))$ can be straightforwardly defined:
\begin{equation}
    H^1(\Gamma(t)) \coloneqq \left\{f\in L^2(\Gamma(t))\, \text{and}\, \ \ud_i f \in L^2(\Gamma(t)),\quad i = 1,2,3\right\},
\end{equation}  
which can also be easily extended to vector-valued function space $[H^1(\Gamma(t))]^3$ and matrix-valued function space  $[H^1(\Gamma(t))]^{3\times 3}$. 

Next, we consider the geometric quantities {on a surface $\Gamma$}. The mean curvature is defined as the trace of the extended Weingarten map of the surface
\begin{equation}\label{def cur}
    \cur \coloneqq\ds\cdot \n = \text{Tr}(\mathbf{A}),
\end{equation}
where $\mathbf{A} := \ds \n$ is also called as the extended Weingarten map, which satisfies $\mathbf{A} = \mathbf{A}^T$. The mean curvature $\cur$ satisfies the following mean curvature vector identity \cite{Barrett2020}, which is frequently used in the BGN schemes:
\begin{equation}\label{BGN curvature}
    \cur \n = - \lap \X.
\end{equation}
The Willmore energy is therefore given by
\begin{equation}
    \label{continuous energy}
    W(t) \coloneqq W(\Gamma(t)) = \frac{1}{2}\int_{\Gamma(t)} \cur^2 \, \dd A = \frac{1}{2}\left(\cur, \cur \right)_{\Gamma(t)}.
\end{equation}

Using the above notation, we are able to write the Willmore flow in 3D as the following coupled fourth-order nonlinear parabolic geometric partial differential equations (PDEs): 
\begin{subequations}
    \label{original Willmore formulation}
    \begin{align}
        &\dt \X =V\n, \label{original Willmore formulation X}\\
        &V = \lap\cur+ |\mathbf{A}|^2 \cur -\frac{1}{2} \cur^3,\label{original Willmore formulation V}\\
        & \cur \n = - \lap \X,\label{original Willmore formulation H}
    \end{align}
\end{subequations}
where $|\mathbf{A}|^2\coloneqq \text{Tr}(\mathbf{A}^T\mathbf{A})$, with the initial condition $\X(\cdot,0)=\X_0(\cdot)$.


\subsection{New geometric identities}
In this subsection, we focus on establishing two geometric identities that are crucial for constructing the ES-PFEM for the Willmore flow in 3D. The first identity characterizes the normal velocity vector of the Willmore flow, while the second one provides a novel expression for the mean curvature evolution that holds for all geometric flows. In contrast to our previous work on the planar Willmore flow \cite{bao2025energy}, both identities are formulated in the variational 
(or weak) forms.

\begin{lemma}\label{lemma V}
    For a solution $\X$ of the parameterized Willmore flow \eqref{original Willmore formulation} with the scalar normal velocity $V$, it holds that 
    \begin{align}
        \label{new identity for V}
        &\int_{\Gamma(t)} V\n\cdot \w \, \dd A \nonumber \\
        & = \int_{\Gamma(t)} \left(\cur \mathbf{A}  -\n (\ds \cur)^T - \frac{1}{2}\cur^2\ds \X\right): \ds \w \, \dd A, \quad \forall\, \w \in [H^1(\Gamma(t))]^3.
    \end{align}
\end{lemma}
\begin{proof}
    By utilizing \eqref{divergence vector vector} with $\boldsymbol{f} = \n$ and $\boldsymbol{g} = \ds \cur$, \eqref{divergence scalar matrix} with $f = \cur$ and $\boldsymbol{G} = \mathbf{A}$, and \eqref{divergence scalar matrix} with $f = \cur^2$ and $\boldsymbol{G} = \ds \X$,  we can derive
    \begin{subequations}
        \begin{align}
            &\ds \cdot(\n (\ds \cur)^T) = \ds\n \ds\cur + \lap \cur \n = \mathbf{A} \ds \cur + \lap \cur \n,  \label{pf1 eq1}\\
            &\ds \cdot (\cur \ds \n) = \ds\n \ds\cur+ \cur \lap \n = \mathbf{A} \ds \cur + \cur \lap \n,  \label{pf1 eq2} \\
            & \ds \cdot (\cur^2 \ds \X) = 2\cur \ds \X \ds \cur + \cur^2 \lap \X = 2\cur  \ds \cur - \cur^3 \n.   \label{pf1 eq3}
        \end{align}
    \end{subequations}
 The last equality makes use of \eqref{BGN curvature}, $\n \cdot \ds \cur = 0$, and $\ds \X = I_3 - \n \n^T$ from \cite[(1)]{Barrett2020}. 
 Using \eqref{gradient multi scalar}, \eqref{pf1 eq1}, and \eqref{pf1 eq2}, there holds
 \begin{align}\label{pf1 eq4}
     \lap (\cur \n) = \ds\cdot \left( \n(\ds \cur)^T + \cur \ds \n\right) 
     =\,  2\mathbf{A} \ds\cur+ \cur \lap \n +  \lap\cur \n.
 \end{align}
Moreover, the expression for the surface gradient related to the mean curvature is given as follows (see Lemma 16 of \cite{Barrett2020}):
\begin{align}\label{pf1 eq5}
    \ds \cur = \lap \n + |\mathbf{A}|^2 \n.
\end{align}
Recalling $V = \lap\cur + |\mathbf{A}|^2\cur - \frac{1}{2}\cur^3$, combining \eqref{pf1 eq2}, \eqref{pf1 eq3}, \eqref{pf1 eq4} and \eqref{pf1 eq5}, we deduce
\begin{align}
    &\lap(\cur\n)- \ds\cdot \left(2\cur \mathbf{A} - \frac{1}{2}\cur^2 \ds\X\right)\nonumber\\
    &\ = \left(2\mathbf{A} \ds\cur+ \cur \lap \n +  \lap\cur \n\right) -2(\mathbf{A} \ds\cur + \cur \lap \n) + \left(\cur \ds\cur -\frac{1}{2}\cur^3 \n\right)\nonumber\\
    &\ =- \cur \lap \n + \lap\cur \n +\cur \ds\cur -\frac{1}{2}\cur^3 \n \nonumber\\
    &\ =  \left( \lap \cur + |\mathbf{A}|^2 \cur  - \frac{1}{2}\cur^3  \right)\n \nonumber\\
    &\ =V\n.
\end{align}
Finally, multiplying both sides of the above equation by $\w$,  then integrating over $\Gamma(t)$, and applying the product rule \eqref{divergence matrix vector}, \eqref{gradient multi scalar}, and the divergence theorem given in \cite[Theorem 21]{Barrett2020} with $\boldsymbol{f} = \left(2\cur \mathbf{A} - \frac{1}{2}\cur^2 \ds\X\right)^T \boldsymbol{\omega} = \left(2\cur \mathbf{A} - \frac{1}{2}\cur^2 \ds\X\right) \boldsymbol{\omega}$, we have
\begin{align}
    \int_{\Gamma(t)}V\n \cdot \w\, \dd A=& \int_{\Gamma(t)}\left( \lap(\cur\n)- \ds\cdot \left(2\cur \mathbf{A} - \frac{1}{2}\cur^2 \ds\X\right) \right)\cdot \w\,\dd A\nonumber\\
    =&\int_{\Gamma(t)} \left(-\ds(\cur\n)+ \left(2\cur \mathbf{A}- \frac{1}{2}\cur^2 \ds\X\right)\right):\ds\w\,\dd A\nonumber\\
    & \qquad -\int_{\Gamma(t)} \cur\n^T(2\cur \mathbf{A}- \frac{1}{2}\cur^2 \ds\X) \w\,\dd A\nonumber\\
    =&\int_{\Gamma(t)} \left(-\n(\ds\cur)^T-\cur \mathbf{A} + \left(2\cur \mathbf{A} - \frac{1}{2}\cur^2 \ds\X\right)\right):\ds\w\,\dd A\nonumber\\
    =&\int_{\Gamma(t)} \left(\cur\mathbf{A} -\n (\ds\cur)^T- \frac{1}{2}\cur^2\ds\X\right):\ds\w\,\dd A,
\end{align}
which yields the result \eqref{new identity for V}.\qed
\end{proof}

In addition, the integral geometric identity related to the temporal evolution of the mean curvature is given by the following lemma. For notational convenience, we always use $\partial_t$ to represent the material derivative in \cite{Barrett2020}.
\begin{lemma}\label{lemma H}
    For the mean curvature $\cur$ of the evolving surface $\Gamma(t)$ parameterized by $\X(\cdot, t)$, there holds
    \begin{equation}
        \label{new identity for H}
        \int_{\Gamma(t)} \dt \cur  \varphi \,\dd A = \int_{\Gamma(t)}\ds (\dt\X): \left(\n(\ds\varphi)^T- \varphi \mathbf{A}\right)\,\dd A, \quad \forall \varphi \in H^1(\Gamma(t)).
    \end{equation}

\end{lemma}
\begin{proof}
From the Huisken's evolution equation for the mean curvature \cite{huisken1984flow}, we know that 
\begin{equation}
    \label{pf2 eq1}
    \dt \cur = -\lap V - |\mathbf{A}|^2 V.
\end{equation}
In our parameterization, $\dt \X =V\n$, thus \eqref{gradient multi scalar} yields
\begin{equation}
    \label{pf2 eq2}
    \ds(\dt\X) = \ds(V\n) = \n (\ds V)^T + V \mathbf{A}.
\end{equation}
Furthermore, noting $\mathbf{A}^T\n = \boldsymbol{0}$, by left-multiplying both sides of \eqref{pf2 eq2} with $\n^T$, it turns out
\begin{equation}
    \label{pf2 eq3}
    (\ds V)^T = \n^T \ds (\dt \X). 
\end{equation}
By multiplying both sides of \eqref{pf2 eq1} with $\varphi$, then integrating over $\Gamma(t)$, and applying the above equation, we can proceed directly to the computations:
\begin{align}
    \int_{\Gamma(t)} \dt \cur \varphi \, \dd A &= \int_{\Gamma(t)} (-\lap V -|\mathbf{A}|^2V)\, \varphi \dd A\nonumber\\
    & = \int_{\Gamma(t)} \ds V \cdot \ds \varphi  \dd A - \int_{\Gamma(t)} \text{Tr}\left(\mathbf{A}^TV \mathbf{A}\right)\varphi\, \dd A\nonumber\\
    & = \int_{\Gamma(t)} (\ds V)^T \ds \varphi  \dd A - \int_{\Gamma(t)} \text{Tr}\left(\mathbf{A}^T(V \mathbf{A} + \n (\ds V)^T)\right)\varphi\, \dd A\nonumber\\
    & = \int_{\Gamma(t)} \n^T\ds(\dt\X)\ds \varphi  \dd A - \int_{\Gamma(t)} \text{Tr}(\mathbf{A}^T\ds(\dt \X))\varphi\, \dd A\nonumber\\
    & = \int_{\Gamma(t)}\ds (\dt\X): (\n(\ds\varphi)^T- \varphi \mathbf{A})\,\dd A,
\end{align}
which validates \eqref{new identity for H}.\qed
\end{proof}
\subsection{A variational formulation and its properties}

%

Let $(\mathbf{X}(\cdot,t), V(\cdot, t), \cur(\cdot,t)) \in [H^1(\Gamma(t))]^3 \times H^1(\Gamma(t)) \times H^1(\Gamma(t))$ be a solution to the Willmore flow \eqref{original Willmore formulation} under our parameterization. We reformulate the governing equation \eqref{original Willmore formulation X} as $\n \cdot \dt \X = V$. Then we multiply by test functions $\phi \in L^2(\Gamma(t))$ and integrate over the surface $\Gamma(t)$ to get its variational form. Furthermore, the variational forms of the normal velocity $V(\cdot, t)$ and the mean curvature $\cur(\cdot, t)$ are given by Lemma \ref{lemma V} and Lemma \ref{lemma H}, respectively. Consequently, this derivation yields a consistent variational formulation for the Willmore flow given by
\begin{subequations}\label{new weak form}
\begin{align}
    &\left(\n \cdot \dt \X, \phi\right)_{\Gamma(t)} = \left(V, \phi\right)_{\Gamma(t)}, \quad \forall \phi\in H^1(\Gamma(t)),\label{new weak form X}\\
    &\left(V\n, \w\right)_{\Gamma(t)} = \Big\langle\cur \mathbf{A} - \n\left(\ds\cur\right)^T- \frac{1}{2}\cur^2\ds \X, \ds\w\Big\rangle_{\Gamma(t)}, \forall \w\in [H^1(\Gamma(t))]^3,\label{new weak form V}\\
    &\left(\dt\cur, \varphi\right)_{\Gamma(t)} = \big\langle\ds(\dt \X), \n (\ds \varphi)^T- \varphi \mathbf{A} \big\rangle_{\Gamma(t)}, \quad \forall \varphi \in H^1(\Gamma(t));\label{new weak form H}
\end{align}
with the initial data $\X(\cdot,0):=\X_0(\cdot)$ and $\cur(\cdot,0):=-\n\cdot\lap \X_0(\cdot)$.
\end{subequations}
\begin{remark}
    The new variational formulation above is in fact a natural generalization of Bao and Li's work in the two-dimensional case \cite{bao2025energy}. By replacing the surface gradient operator $\ds$ by the arc length derivative $\partial_s$ and utilizing the identity $\partial_s \n = \kappa\partial_s \X$ with $\kappa$ denoting the mean curvature in 2D, we can recover the variational formulation in the two-dimensional case presented in \cite{bao2025energy}. 
\end{remark}

\begin{lemma}[A new transport theorem]\label{transport thoerem}
    Let a scalar-valued function $f$ be defined on the evolving surface $\Gamma(t)$ and $f$ is differentiable with respect to $t$. Then it holds
    \begin{equation}
        \label{eqn:transport thm}
        \frac{\dd}{\dd t} \int_{\Gamma(t)} f \,\dd A = \int_{\Gamma(t)} \dt f \, \dd A + \int_{\Gamma(t)} f \ds\X :\ds(\dt\X)\, \dd A.
    \end{equation}
\end{lemma}
\begin{proof}
    From (v) in \cite[Remark 6]{Barrett2020} and the fact that $\ds \X = (I_3 - \n \n^T)$, we know that
    \begin{align}\label{pf3 eq2}
        \ds\cdot (\dt\X) = \text{Tr}(\ds(\dt\X))=  \text{Tr}(\ds(\dt\X)(\ds\X)^T)= \ds \X :\ds (\dt\X).
    \end{align}
    Therefore, using the transport theorem \cite[Theorem 32]{Barrett2020}, we can derive that {
    \begin{align}
        \frac{\dd}{\dd t} \int_{\Gamma(t)}f \, \dd A &= \int_{\Gamma(t)} \dt f \, \dd A + \int_{\Gamma(t)} f \ds \cdot \left(\dt \X\right) \, \dd A\nonumber\\
        & = \int_{\Gamma(t)} \dt f \, \dd A + \int_{\Gamma(t)} f \ds \X :\ds(\dt\X)\, \dd A,
    \end{align}}
    which implies \eqref{eqn:transport thm}.
    \qed
\end{proof}
\begin{remark}
    The transport theorem reveals that energy dissipation decomposes into two components: energy change (represented by $\int_{\Gamma(t)} \dt f \, \dd A$) and domain change (represented by $\int_{\Gamma(t)} f \ds\X :\ds(\dt\X)\, \dd A$). Instead of $\ds \cdot  (V \n)$, our new formulation adopts $\ds\X:\ds(\dt\X)$ to represent the domain change term, which is inspired by the fundamental inequality from the BGN scheme in the discrete setting as \cite{barrett2007variational}:
    \begin{equation*}
    \int_{\Gamma^m} \ds \X^{m+1} :\ds ( \X^{m+1} - \X^m) \, \dd A \geq |\Gamma^{m+1}| - |\Gamma^m|, \qquad m\ge0.
    \end{equation*}
    This reformulation is crucial for enabling energy stability in the full-discretization.
\end{remark}

Now we can turn back to prove the energy dissipative property of the new variational formulation \eqref{new weak form}.
\begin{theorem}[Energy dissipation]\label{thm:continuous energy dissipation }
{Let $\left(\X(\cdot,t), V(\cdot, t), \cur(\cdot,t)\right)$ be a sufficiently smooth solution of the new variational formulation \eqref{new weak form} with the initial data $\X(\cdot,0)$ and $\cur(\cdot,0)$.}  Then the Willmore energy $\Gamma(t)$ defined in \eqref{continuous energy} is decreasing over time, i.e.
\begin{equation}
\label{eqn: continuous energy dissipation}
    W(t) \leq W(t')\leq W(0) = \frac{1}{2}\left(\cur_0,\cur_0\right)_{\Gamma(0)}, \quad \forall t \geq t'\geq 0.
\end{equation}
\end{theorem}
\begin{proof}
We employ Lemma \ref{transport thoerem} for $f = \frac{1}{2}\cur^2$ and derive that
\begin{align}\label{pf4 eq1}
    \frac{\dd}{\dd t} W(t) &= \frac{\dd}{\dd t}\int_{\Gamma(t)}\frac{\cur^2}{2}\, \dd A \nonumber\\
    &= \int_{\Gamma(t)} \cur \dt \cur\, \dd A + \frac{1}{2}\int_{\Gamma(t)} \cur^2 \ds\X :\ds(\dt\X)\, \dd A,\nonumber\\
    & = \left(\dt \cur, \cur \right)_{\Gamma(t)}+ \Big\langle\frac{\cur^2}{2}\ds\X, \ds(\dt \X)\Big\rangle_{\Gamma(t)}, \qquad t\ge0.
\end{align}  
In the meanwhile, by taking test functions $\phi = -V \in L^2(\Gamma(t))$ in \eqref{new weak form X}, $\w = \dt \X \in [H^1(\Gamma(t))]^3$ in \eqref{new weak form V} and  $\varphi = \cur \in H^1(\Gamma(t))$ in \eqref{new weak form H}, we obtain
\begin{subequations}
    \begin{align}
        &-\left(\n \cdot \dt \X, V\right)_{\Gamma(t)} = -\left(V, V\right)_{\Gamma(t)}, \label{continuous, test function X}\\
        &\left(V\n, \dt \X\right)_{\Gamma(t)} = \Big\langle\cur \mathbf{A} - \n\left(\ds\cur\right)^T- \frac{1}{2}\cur^2\ds \X, \ds (\dt \X)\Big\rangle_{\Gamma(t)}, \label{continuous, test function V}\\
        &\left(\dt\cur, \cur\right)_{\Gamma(t)} = \big\langle\ds(\dt \X), \n (\ds \cur)^T- \cur \mathbf{A} \big\rangle_{\Gamma(t)}, \qquad t\ge0. \label{continuous, test function H}
    \end{align}
\end{subequations}
Summing up the above three equations, we get
\begin{equation}\label{pf4 eq2}
    \begin{aligned}
    \left(\dt\cur, \cur\right)_{\Gamma(t)} & = - \left(V, V\right)_{\Gamma(t)}-\Big\langle\frac{\cur^2}{2}\ds\X, \ds(\dt \X)\Big\rangle_{\Gamma(t)},\qquad t\ge0.
    \end{aligned}
\end{equation}
Utilizing \eqref{pf4 eq1} and \eqref{pf4 eq2}, it can be deduced that
\begin{equation}
    \frac{\dd}{\dd t} W(t) = -\left(V, V\right)_{\Gamma(t)} \leq 0, \qquad t\ge0,
\end{equation}
which yields the energy dissipative property \eqref{eqn: continuous energy dissipation}.
\qed
\end{proof}

\begin{remark}
    In the above proof of the energy dissipation, by employing \eqref{new weak form H} along with the new transport theorem Lemma \ref{transport thoerem}, we circumvent the treatment of the high-order term $\lap V$ as in \eqref{pf2 eq1}, and the extended Weingarten map $\mathbf{A}$ cancels out between \eqref{continuous, test function V} and \eqref{continuous, test function H}. This will help us develop the energy-stable numerical scheme under piecewise linear finite elements in the following section.
\end{remark}

\section{An energy-stable parametric finite element method}
In this section, by utilizing the parametric finite element method in spatial discretization and the backward Euler method in temporal discretization, we present an ES-PFEM for the Willmore flow in 3D by building upon the new variational formulation \eqref{new weak form}. We also prove rigorously the unconditional energy stability of the proposed ES-PFEM.

\subsection{A full-discretization}
To derive an appropriate fully discrete scheme, we introduce the following parametric finite element approximation. Let $\tau > 0$ be the uniform time step size and denote the discrete time levels as $t_m = m\tau$ for $m = 0,1,2,\cdots$. At every time step $t_m$, we approximate the evolving surface $\Gamma(t_m)$ by the polygonal surface denoted by $\Gamma^m$, which is composed of a collection of mesh vertices $\{ \q_k^m\}_{k = 1}^K$ or equivalently, a collection of nonoverlapping and nondegenerate triangles $\{ \sigma_{j}^m\}_{j = 1}^J$. The polygonal surface can be written as 
$$
\Gamma^m \coloneqq \bigcup_{j = 1}^J \overline{\sigma_{j}^m},\quad  m = 0,1,2,\cdots. 
$$
For $1\leq j \leq J$, the triangle surface element $\sigma_j^m = \{\q_{j_1}^m, \q_{j_2}^m, \q_{j_3}^m\}$ is determined by its three vertices $\q_{j_1}^m, \q_{j_2}^m, \q_{j_3}^m$ in counter-clockwise order on the outer surface of $\Gamma^m$ in 3D. 

We use the continuous piecewise linear elements in space to approximate the space $L^2(\Gamma(t_m))$, which is given as 
\begin{equation}
    \mathbb{K}^m \coloneqq \left\{ u \in C(\Gamma^m)\, \Big|\, \,  u |_{\sigma_j^m}\in \mathcal{P}^1(\sigma_j^m), \quad \forall 1\leq j \leq J \right\},
\end{equation}
where $\mathcal{P}^1(\sigma_j^m)$ denotes the function space of all polynomials with degree at most $1$ on $\sigma_{j}^m$. Let $\n^m$ be the outward unit normal vector on polygonal surface $\Gamma^m$, then it is a piecewise constant vector-valued function which can be defined via
\begin{align} \label{discrete normal vector}
    \n^m|_{\sigma_j^m}\coloneqq \n^m_j = \frac{\mathcal{J}\{\sigma_j^m\}}{|\mathcal{J}\{\sigma_j^m\}|},
\end{align}
where $\mathcal{J}\{\sigma_j^m\}$ is the direction vector given by the three vertices denoted by 
\begin{equation}\label{direction vector}
    \mathcal{J}\{ \sigma_j^m\} =  \left(\q_{j_2}^m- \q_{j_1}^m\right)\times  \left(\q_{j_3}^m- \q_{j_1}^m\right),
\end{equation}
on each $\sigma_j^m = \{\q_{j_1}^m, \q_{j_2}^m, \q_{j_3}^m\}$.

{For scalar-valued or vector-valued functions $u^h,v^h$ on $\Gamma^m$ with values in $\mathbb R$ or $\mathbb R^3$}, the $L^2$ inner product $(\cdot,\cdot )_{\Gamma(t_m)}$ is approximated by the mass-lumped inner product $(\cdot, \cdot)_{\Gamma^m}^h$,
\begin{equation}
    \left(u^h, v^h\right)_{\Gamma^m}^h\coloneqq \frac{1}{3}\sum_{j = 1}^J \sum_{k = 1}^3|\sigma_j^m|\, u^h\left( (\q_{j_k}^m)^{-}\right)\cdot v^h\left( (\q_{j_k}^m)^{-}\right),
\end{equation}
where $|\sigma_j^m| = \frac{1}{2}|\mathcal{J}\{\sigma_j^m\}|$ is the area of $\sigma_j^m$, and  $f((\q_{j_k}^m)^{-}) = \lim\limits_{\boldsymbol{x} \rightarrow \q_{j_k}^m, \boldsymbol{x}\in \sigma_j^m} f(\boldsymbol{x})$. {For matrix-valued functions $\boldsymbol{U}^h, \boldsymbol{V}^h: \Gamma^m \to \mathbb R^{3\times 3}$ }, we approximate the inner product \eqref{inner product matrix} by $\left< \boldsymbol{U}^h, \boldsymbol{V}^h\right>_{\Gamma^m}^h$ defined as
\begin{equation}
    \left<\boldsymbol{U}^h, \boldsymbol{V}^h\right>_{\Gamma^m}^h \coloneqq\frac{1}{3} \sum_{j = 1}^J \sum_{k =1}^3 |\sigma_j^m|\, \boldsymbol{U}^h\left(  (\q_{j_k}^m)^{-}\right):\boldsymbol{V}^h\left(  (\q_{j_k}^m)^{-}\right).
\end{equation}
{Finally, the surface gradient operator $\ds$ on a triangle surface element $\sigma = \{\q_{1}, \q_{2}, \q_{3}\}$ for any $f\in \mathbb{K}^m$ is given as in \cite{bao2021structure} by} 
\begin{equation}\label{discretized surface gradient}
    (\ds f)\big|_{\sigma} \coloneqq  f(\q_1)\frac{(\q_2-\q_3)\times \n}{|\mathcal{J}\{\sigma\}|}+f(\q_2)\frac{(\q_3-\q_1)\times \n}{|\mathcal{J}\{\sigma\}|}+f(\q_3)\frac{(\q_1-\q_2)\times \n}{|\mathcal{J}\{\sigma\}|},
\end{equation}
where $\n$ is the normal vector of $\sigma$ given by \eqref{discrete normal vector}.

Based on the finite element space $\mathbb{K}^m$ defined above, we can define a
 parameterization over $\Gamma^m$ for $\Gamma^{m+1}$ as $\X^{m+1}(\cdot) \in [\mathbb{K}^m]^3$, where $[\mathbb{K}^m]^3$ can be viewed as a finite element approximation of the function space $[H^1(\Gamma(t_m))]^3$, {thus the surface $\Gamma^{m+1}$ is defined through $\q_k^{m+1} = \boldsymbol{X}^{m+1}(\boldsymbol{q}_k^m)$ and $\sigma_k^{m+1} = \boldsymbol{X}^{m+1}(\sigma_k^m)$. In particular, $\X^m(\cdot)$ can represent the identity function in $[\mathbb{K}^m]^3$.} 
 
{Moreover, let $\cur^{m+1}, V^{m+1}\in \mathbb{K}^m$ denote the numerical approximations of the mean curvature and scalar normal velocity of $\Gamma^{m+1}$, respectively. The corresponding push-forward function for $\cur^{m+1}$ is
$\cur^{m+1}_{\Gamma^{m+1}} = \cur^{m+1}\circ (\X^{m+1})^{-1}\in \mathbb K^{m+1}$. In the following, for simplicity, we do not distinguish between $\cur^{m+1}$ and $\cur^{m+1}_{\Gamma^{m+1}}$, and denote both simply by $\cur^{m+1}$.}

Then an unconditionally energy-stable PFEM based on the implicit backward Euler scheme is proposed as follows: Given the initial data of a polygonal surface $\Gamma^0$: \quad $\X^0 (\cdot) \in [\mathbb{K}^0]^3$ and $\cur^0(\cdot)\in \mathbb{K}^0$, find the solution $\left(\X^{m+1},V^{m+1},\cur^{m+1}\right) \in [\mathbb{K}^m]^3 \times \mathbb{K}^m\times \mathbb{K}^m$ such that
\begin{subequations}
\label{eqn:full wm}
\begin{align}
\label{eqn:full1 wm}
&\left(\frac{\X^{m+1}-\X^m}{\tau}\cdot \n^m, \phi^h\right)_{\Gamma^m}^h= \left(V^{m+1},\,  \phi^h\right)_{\Gamma^m}^h,\qquad \qquad \quad  \forall \phi^h\in \mathbb{K}^m,\\[0.5em]
\label{eqn:full2 wm}
&\left(V^{m+1}\n^m,\w^h\right)_{\Gamma^m}^h=\Big\langle\cur^{m+1} \mathbf{A}^{m}-\n^m(\ds \cur^{m+1})^T-\frac{1}{2}(\cur^{m+1})^2 \ds \X^{m+1},\, \ds \w^h\Big\rangle_{\Gamma^m}^h,\nonumber\\
& \qquad \qquad \qquad \qquad\qquad \qquad \qquad \forall \w^h  \in [\mathbb{K}^m]^3, \\[0.5em]
\label{eqn:full3 wm}
&\left(\cur^{m+1}-\cur^m, \,\varphi^h\right)_{\Gamma^m}^h = \Big\langle \ds(\X^{m+1}-\X^m), \n^m (\ds\varphi^h)^T -\varphi^h \mathbf{A}^{m} \Big\rangle_{\Gamma^m}^h,\quad \forall \varphi^h \in \mathbb{K}^m,
\end{align}
\end{subequations}
where $\n^m$ is given in \eqref{discrete normal vector} and $\mathbf{A}^{m}$ is a numerical approximation of the extended Weingarten map $\ds\n$ at $t=t_m$. 
Since $\n^m$ is piecewise constant in $[\mathbb{K}^h]^3$, the extended Weingarten map $\ds\n^m$ vanishes under the surface gradient \eqref{discretized surface gradient}, necessitating a refined approximation. Here we employ vertex unit normals $\boldsymbol{w}^m$ (see \cite{barrett2008parametric2}) to make normals non-constant over each surface element $\sigma_{j}^m$, enabling a direct computation of the discrete extended Weingarten map as
\begin{align}\label{eqn:A^m}
    \mathbf{A}^{m} = \ds \boldsymbol{w}^m, \qquad m\ge0.
\end{align}

\subsection{Energy stability}
For the polygonal surface $\Gamma^m$, we define the discretized Willmore energy as 
{
\begin{equation}
    \label{discretized Willmore energy}
    W^m \coloneqq  \frac{1}{2}\left(\cur^m, \cur^m\right)_{\Gamma^m}^h = \frac{1}{6}\sum_{j = 1}^J\sum_{k = 1}^3|\sigma_{j}^m|\left( \cur^m(\q_{j_k}^m)\right)^2, \qquad m\ge0.
\end{equation}}
{
To establish the energy stability of the proposed scheme, we first recall a
local geometric estimate on each surface element. This estimate is essentially
Lemma 2.1 in \cite{barrett2008parametric1}, but rewritten in the notation of the
present paper. For completeness, we include the proof.

\begin{lemma}\label{lemma: local estimate}
Let $\Gamma^m = \bigcup_{j=1}^{J} \overline{\sigma_j^m}$ be a polyhedral surface,
and let $\X^{m+1} \in [\mathbb{K}^m]^3$ be the map defining the
updated surface $\Gamma^{m+1} = \X^{m+1}(\Gamma^m)$, with $\X^m$ being the identity map on $\Gamma^m$.
Then, for each surface element $\sigma_j^m$, the following estimates hold:
\begin{align}
    &\frac{1}{2} |\sigma_j^m| \left( \ds \X^{m+1}\big|_{\sigma_j^m} : \ds \X^{m+1}\big|_{\sigma_j^m} \right) \geq |\sigma_j^{m+1}|, \label{lemma: local estimate 2} \\[0.5em]
    &\frac{1}{2} |\sigma_j^m| \left( \ds \X^{m}\big|_{\sigma_j^m} : \ds \X^{m}\big|_{\sigma_j^m} \right) = |\sigma_j^{m}|, \label{lemma: local estimate 1}
\end{align}
where $\sigma_j^{m+1} = \X^{m+1}(\sigma_j^m)$.
\end{lemma}
\begin{proof}
We fix an element 
$\sigma_j^m = \{\q_{j_1}^m, \q_{j_2}^m, \q_{j_3}^m\}$, denote 
$\boldsymbol{a}^m = \q_{j_2}^m - \q_{j_1}^m$ and $\boldsymbol{b}^m = \q_{j_3}^m - \q_{j_1}^m$. 
Then, the updated surface $\Gamma^{m+1}$ is defined through 
$\q_{j_k}^{m+1} = \X^{m+1}(\q_{j_k}^{m})$ and 
$\sigma_j^{m+1} = \X^{m+1}(\sigma_j^m)$. Recalling \eqref{direction vector}, the areas of $\sigma_j^m$ and $\sigma_j^{m+1}$ are
\begin{equation}\label{proof of lemma: area at t_m}
|\sigma_j^m| = \frac12 |\mathcal{J}\{\sigma_j^m\}| = \frac12 |\boldsymbol{a}^m \times \boldsymbol{b}^m|,
\end{equation}
and 
\begin{equation}\label{proof of lemma: area at t_m+1}
|\sigma_j^{m+1}| = \frac12 |\mathcal{J}\{\sigma_j^{m+1}\}| = \frac12 |\boldsymbol{a}^{m+1} \times \boldsymbol{b}^{m+1}|.
\end{equation}
Moreover, by \eqref{discretized surface gradient}, on the element $\sigma_j^m$,
\begin{align}
\ds \X^{m+1}\big|_{\sigma_j^m} 
&= \q_{j_1}^{m+1}\otimes \frac{(\q_{j_2}^m-\q_{j_3}^m)\times \n_j^m}{|\boldsymbol{a}^m \times \boldsymbol{b}^m|}
+ \q_{j_2}^{m+1}\otimes \frac{(\q_{j_3}^m-\q_{j_1}^m)\times \n_j^m}{|\boldsymbol{a}^m \times \boldsymbol{b}^m|} \nonumber\\
&\quad + \q_{j_3}^{m+1}\otimes \frac{(\q_{j_1}^m-\q_{j_2}^m)\times \n_j^m}{|\boldsymbol{a}^m \times \boldsymbol{b}^m|} \nonumber\\[0.5em]
&= \boldsymbol{a}^{m+1}\otimes \frac{\boldsymbol{b}^m \times \n_j^m}{|\boldsymbol{a}^m \times \boldsymbol{b}^m|}
- \boldsymbol{b}^{m+1}\otimes \frac{\boldsymbol{a}^m \times \n_j^m}{|\boldsymbol{a}^m \times \boldsymbol{b}^m|}.
\end{align}
Hence, by Lagrange's identity 
$
(\boldsymbol{c}\times\boldsymbol{d})\cdot(\boldsymbol{e}\times\boldsymbol{f})
=
(\boldsymbol{c}\cdot\boldsymbol{e})(\boldsymbol{d}\cdot\boldsymbol{f})
-
(\boldsymbol{c}\cdot\boldsymbol{f})(\boldsymbol{d}\cdot\boldsymbol{e})
$, $\boldsymbol{a}^m \cdot \n_j^m= \boldsymbol{b}^m\cdot \n_j^m =0, \n_j^m \cdot \n_j^m = 1$, and identity $|\boldsymbol{c}|^2 |\boldsymbol{d}|^2 = (\boldsymbol{c} \cdot \boldsymbol{d})^2 + |\boldsymbol{c} \cdot \boldsymbol{d}|^2$, one can calculate
\begin{align}\label{proof of lemma: local estimate}
&\ds \X^{m+1}|_{\sigma_j^m}:\ds \X^{m+1}|_{\sigma_j^m} \nonumber \\[0.5em]
&= \frac{|\boldsymbol{b}^m \times \boldsymbol{n}^m_j|^2 |\boldsymbol{a}^{m+1}|^2 + |\boldsymbol{a}^m\times \boldsymbol{n}^m_j|^2 |\boldsymbol{b}^{m+1}|^2
- 2 (\boldsymbol{a}^m\times \boldsymbol{n}^m_j)\cdot (\boldsymbol{b}^m \times \boldsymbol{n}^m_j) (\boldsymbol{a}^{m+1}\cdot \boldsymbol{b}^{m+1})}{|\boldsymbol{a}^m \times \boldsymbol{b}^m|^2} \nonumber\\[0.5em]
&= \frac{|\boldsymbol{b}^m|^2 |\boldsymbol{a}^{m+1}|^2 + |\boldsymbol{a}^m|^2 |\boldsymbol{b}^{m+1}|^2
- 2 (\boldsymbol{a}^m\cdot \boldsymbol{b}^m)(\boldsymbol{a}^{m+1}\cdot \boldsymbol{b}^{m+1})}{|\boldsymbol{a}^m \times \boldsymbol{b}^m|^2} \nonumber\\[0.5em]
& \geq \frac{2\sqrt{(\boldsymbol{a}^m\cdot \boldsymbol{b}^m)^2 + |\boldsymbol{a}^m \times \boldsymbol{b}^m|^2}\sqrt{(\boldsymbol{a}^{m+1}\cdot \boldsymbol{b}^{m+1})^2 + |\boldsymbol{a}^{m+1} \times \boldsymbol{b}^{m+1}|^2}}{|\boldsymbol{a}^m \times \boldsymbol{b}^m|^2} \nonumber\\[0.5em]
& \qquad - \frac{2 (\boldsymbol{a}^m\cdot \boldsymbol{b}^m)(\boldsymbol{a}^{m+1}\cdot \boldsymbol{b}^{m+1})}{|\boldsymbol{a}^m \times \boldsymbol{b}^m|^2} \nonumber\\[0.5em]
&\ge\frac{2 |\boldsymbol{a}^m \times \boldsymbol{b}^m||\boldsymbol{a}^{m+1} \times \boldsymbol{b}^{m+1}|}{|\boldsymbol{a}^m \times \boldsymbol{b}^m|^2} =2 \frac{|\boldsymbol{a}^{m+1}\times \boldsymbol{b}^{m+1}|}{|\boldsymbol{a}^m \times \boldsymbol{b}^m|}.
\end{align}
The last inequality follows from the Cauchy inequality
$\sqrt{a^2+b^2}\sqrt{c^2 + d^2}\ge (ac + bd)$.
Combining with \eqref{proof of lemma: area at t_m} and \eqref{proof of lemma: area at t_m+1} gives \eqref{lemma: local estimate 2}.

For the identity \eqref{lemma: local estimate 2}, one computes
\begin{align}
\ds \X^m|_{\sigma_j^m}:\ds \X^m|_{\sigma_j^m}
&= \frac{2 |\boldsymbol{a}^m|^2 |\boldsymbol{b}^m|^2 - 2 (\boldsymbol{a}^m\cdot \boldsymbol{b}^m)^2}{|\boldsymbol{a}^m \times \boldsymbol{b}^m|^2} \nonumber\\
&= 2\frac{ |\boldsymbol{a}^m \times \boldsymbol{b}^m|^2}{|\boldsymbol{a}^m \times \boldsymbol{b}^m|^2} = 2,
\end{align}
which immediately implies \eqref{lemma: local estimate 1}.
\end{proof}
}
Then, we have the following energy stability for the proposed ES-PFEM \eqref{eqn:full wm}.
\begin{theorem}[Unconditional energy dissipation of \eqref{eqn:full wm}]
Let $\left(\X^{m+1}, V^{m+1} ,\cur^{m+1}\right)$ be a solution of the ES-PFEM \eqref{eqn:full wm} with the initial data $\X^0$ and $\cur^0$. Then it holds that the discretized Willmore energy \eqref{discretized Willmore energy} decreases, i.e.
\begin{equation}\label{fulld stability}
    W^{m+1}\leq W^m \leq W^0 =  \frac{1}{2}\left(\cur^0, \cur^0\right)_{\Gamma^0}^h, \qquad m\ge0.
\end{equation}
\end{theorem}
\begin{proof}
Choose $\phi^h=-\tau V^{m+1} \in \mathbb{K}^m$ in \eqref{eqn:full1 wm}, $\w=\X^{m+1}-\X^m \in [\mathbb{K}^m]^3$ in \eqref{eqn:full2 wm}, and $\varphi^h = \cur^{m+1} \in \mathbb{K}^m$  in \eqref{eqn:full3 wm}, we obtain
\begin{align}
\label{test func 1}
-\left((\X^{m+1}-\X^m)\cdot \n^m, V^{m+1}\right)_{\Gamma^m}^h=-\tau\left(V^{m+1},\, V^{m+1}\right)_{\Gamma^m}^h,
\end{align}
\begin{align}
\label{test func 2}
\left(V^{m+1}\n^m,\X^{m+1}-\X^m\right)_{\Gamma^m}^h=&\Big\langle\cur^{m+1} \mathbf{A}^m-\n^m(\ds \cur^{m+1})^T,\, \ds (\X^{m+1}-\X^m)\Big\rangle_{\Gamma^m}^h\nonumber\\
 &-\Big\langle\frac{1}{2}(\cur^{m+1})^2 \ds \X^{m+1},\, \ds (\X^{m+1}-\X^m)\Big\rangle_{\Gamma^m}^h,
\end{align}
\begin{align}
\label{test func 3}
 &\left(\cur^{m+1}-\cur^m, \,\cur^{m+1}\right)_{\Gamma^m}^h = \Big\langle \ds(\X^{m+1}-\X^m), \n^m (\ds \cur^{m+1})^T -\cur^{m+1} \mathbf{A}^m \Big\rangle_{\Gamma^m}^h.
\end{align}
Summing \eqref{test func 1}, \eqref{test func 2} and \eqref{test func 3}, we derive that
\begin{align}\label{aeq 0}
	&\left(\cur^{m+1}-\cur^m, \,\cur^{m+1}\right)_{\Gamma^m}^h + \Big\langle\frac{1}{2}(\cur^{m+1})^2\ds \X^{m+1},\, \ds\X^{m+1}-\ds \X^m\Big\rangle_{\Gamma^m}^h \nonumber\\
    &=-\tau\left(V^{m+1},\, V^{m+1}\right)_{\Gamma^m}^h \leq 0, \qquad m\ge0.
\end{align}  
By applying the inequality $a(a-b)\geq \frac{1}{2}a^2- \frac{1}{2}b^2$, we know that 
\begin{align}\label{aeq 1}
    \left(\cur^{m+1}-\cur^m, \cur^{m+1}\right)_{\Gamma^m}^h&\geq \frac{1}{2}\left(\cur^{m+1}, \cur^{m+1}\right)_{\Gamma^m}^h -\frac{1}{2}\left(\cur^{m},\cur^{m}\right)_{\Gamma^m}^h\nonumber\\[0.5em]
&= \frac{1}{2}\left(\cur^{m+1}, \cur^{m+1}\right)_{\Gamma^m}^h-W^m.
\end{align}
Similarly, for any $\boldsymbol{M} = (m_{ij})$ and $\boldsymbol{N} = (n_{ij})$ in $\bR^{3\times 3 }$, we have
\begin{align}\label{matrix basic inequality}
\boldsymbol{M}:(\boldsymbol{M}-\boldsymbol{N}) &= \sum_{i = 1}^3\sum_{j = 1}^3m_{ij}(m_{ij} -n_{ij})\geq \frac{1}{2}\sum_{i = 1}^3\sum_{j = 1}^3(m_{ij})^2-\frac{1}{2}\sum_{i = 1}^3\sum_{j = 1}^3(n_{ij})^2\nonumber\\
&=\frac{1}{2}\boldsymbol{M}:\boldsymbol{M}-\frac{1}{2}\boldsymbol{N}:\boldsymbol{N}.
\end{align}
  { In the following, We denote $\cur_{j_k}^{m+1} = \cur^{m+1}(\q_{j_k}^m)$ . Thus, it follows from \eqref{matrix basic inequality} that}
\begin{align}\label{aeq 2}
&\Big\langle\frac{1}{2}(\cur^{m+1})^2\ds \X^{m+1},\, \ds \X^{m+1}-\ds \X^m\Big\rangle_{\Gamma^m}^h\nonumber\\[0.5em]
&\geq \frac{1}{4}\Big\langle(\cur^{m+1})^2\ds \X^{m+1},\, \ds \X^{m+1}\Big\rangle_{\Gamma^m}^h- \frac{1}{4}\Big\langle(\cur^{m+1})^2\ds \X^{m},\, \ds \X^{m}\Big\rangle_{\Gamma^m}^h\nonumber\\[0.5em]
&=\frac{1}{4}\sum_{j = 1}^J|\sigma_j^m|\left(\ds \X^{m+1}\Big|_{\sigma_j^m}: \ds \X^{m+1}\Big|_{\sigma_j^m}\right) \sum_{k = 1}^3\frac{(\cur_{j_k}^{m+1})^2}{3}\nonumber\\
&\ -\frac{1}{4}\sum_{j = 1}^J|\sigma_j^m|\left(\ds \X^m\Big|_{\sigma_j^m}: \ds \X^m\Big|_{\sigma_j^m}\right) \sum_{k = 1}^3\frac{(\cur_{j_k}^{m+1})^2}{3}\nonumber\\[0.5em]
&\geq \frac{1}{2} \sum_{j=1}^J\sum_{k=1}^3\frac{(\cur^{m+1}_{j_k})^2}{3}(|\sigma_j^{m+1}|-|\sigma_j^m|),
\end{align}
{where the last inequality is due to \eqref{lemma: local estimate 2} and \eqref{lemma: local estimate 1} in Lemma \ref{lemma: local estimate}.} Moreover, we have the following identity of $W^{m+1}$,
\begin{align}
    \label{aeq 4}
    W^{m+1}& =  \frac{1}{2}\sum_{j = 1}^J\sum_{k = 1}^3 \frac{(\cur^{m+1}_{j_k})^2}{3} |\sigma_j^{m+1}| \nonumber\\
    & = \frac{1}{2}\sum_{j = 1}^J \sum_{k = 1}^3 \frac{(\cur^{m+1}_{j_k})^2}{3}|\sigma_j^m| +\frac{1}{2}\sum_{j = 1}^J  \sum_{k = 1}^3 \frac{(\cur^{m+1}_{j_k})^2}{3}(|\sigma_j^{m+1}|-|\sigma_j^m|)\nonumber\\
    &=  \frac{1}{2} (\cur^{m+1}, \cur^{m+1})_{\Gamma^m}^h + \frac{1}{2}\sum_{j = 1}^J \sum_{k = 1}^3 \frac{(\cur^{m+1}_{j_k})^2}{3}(|\sigma_j^{m+1}|-|\sigma_j^m|) .
\end{align}
This, together with \eqref{aeq 0}, \eqref{aeq 1} and \eqref{aeq 2}, yields that
\begin{align}\label{aeq 5}
    W^{m+1}-W^m  = &\frac{1}{2} (\cur^{m+1}, \cur^{m+1})_{\Gamma^m}^h - W^m+ \frac{1}{2}\sum_{j = 1}^J  \sum_{k = 1}^3 \frac{(\cur^{m+1}_{j_k})^2}{3}(|\sigma_j^{m+1}|-|\sigma_j^m|)\nonumber\\
    \leq &  \left(\cur^{m+1}-\cur^m, \cur^{m+1}\right)_{\Gamma^m}^h +\Big\langle\frac{1}{2}(\cur^{m+1})^2\ds \X^{m+1},\, \ds \X^{m+1}-\ds \X^m\Big\rangle_{\Gamma^m}^h\nonumber\\[0.5em]
    =& -\tau (V^{m+1}, V^{m+1})_{\Gamma^m}^h \leq 0,
\end{align}
which is the desired energy stability \eqref{fulld stability}.\qed
\end{proof}

\begin{remark}
    In the above proof, we observe that the unconditional energy stability relies on two key ingredients. First, the evolution equation of mean curvature \eqref{new identity for H}, its $\int_{\Gamma(t)} \partial_t \mathcal{H} \phi$ corresponds to the discrete $\left(\cur^{m+1}-\cur^m, \cur^{m+1}\right)_{\Gamma^m}^h$, which controlling the change of curvature. Second, the new transport theorem, its $\int_{\Gamma(t)} \nabla_{\Gamma} \boldsymbol{X}:\nabla_\Gamma \partial_t \boldsymbol{X}$ corresponds to the discrete
    \begin{equation*}
    \Big\langle\ds \X^{m+1},\, \ds \X^{m+1}-\ds \X^m\Big\rangle_{\Gamma^m}^h,
    \end{equation*}
    which captures change of surface analogous to the BGN scheme. Therefore, similar to the BGN scheme \cite{barrett2007parametric}, our method \eqref{new weak form} can be readily extended to higher-order spatial discretizations, see \cite{garcke2025isoparametric}.
\end{remark}

\section{Extension to the curvature-dependent geometric gradient flows}

This section is to extend the ES-PFEM \eqref{eqn:full wm} to the curvature-dependent geometric gradient flows, where the functional $W(\Gamma)$ is generalized to be
\begin{equation}\label{eqn:continuous energy, curvature-dependent geometric gradient flow}
    W(\Gamma) = \int_{\Gamma} f(\cur) \, \dd A,
\end{equation}
where $f(\cur)$ is an arbitrary weakly convex function of the mean curvature $\cur$. The corresponding geometric gradient flow is given by the normal velocity $V$ as \cite{duan2021high}
\begin{equation}\label{eq: curvature-dependent geometric gradient flow}
    V = \lap (f'(\cur)) + |\mathbf{A}|^2f'(\cur) - \cur f(\cur) .
\end{equation}
Specifically, by setting $f(\cur) = \frac{1}{2}\cur^2$, we recover the Willmore energy and the Willmore flow \eqref{eq: Willmore flow}. Furthermore, by setting $f(\cur) \equiv 1$, we have $f'(\cur) \equiv 0$ and $V = -\cur$, which is the mean curvature flow; and by setting $f(\cur) = \cur$, we have $f'(\cur) = 1$ and $V = -2\mathcal{K}$, which is the Gauss curvature flow with $\mathcal{K} = \det(\mathbf{A})$ being the Gaussian curvature.

\subsection{A variational formulation}
Similar to the Willmore flow, in order to derive the variational formulation for the curvature-dependent geometric gradient flow \eqref{eq: curvature-dependent geometric gradient flow}, we need to find a proper weak formulation for $V \n$.  The evolution equation for $\cur$ follows directly from the Willmore flow.

\begin{lemma}\label{lemma: curvature-dependent geometric gradient flow V}
    For a solution $\X$ of the curvature-dependent geometric gradient flow \eqref{eq: curvature-dependent geometric gradient flow} with the scalar normal velocity $V$, it holds that
    \begin{align}
        \label{new identity for V, curvature-dependent geometric gradient flow}
        \int_{\Gamma(t)} V\n\cdot \w \, \dd A&= \int_{\Gamma(t)} \left(f'(\cur) \mathbf{A}  -\n (\ds f'(\cur))^T - f(\cur)\ds \X\right): \ds \w \, \dd A, \nonumber \\
        &\qquad \qquad \quad \forall\, \w \in [H^1(\Gamma(t))]^3.
    \end{align}
\end{lemma}
\begin{proof}
    By utilizing \eqref{divergence vector vector} with $\boldsymbol{f} = \n$ and $\boldsymbol{g} = \ds f'(\cur)$, \eqref{divergence scalar matrix} with $f = f'(\cur)$ and $\boldsymbol{G} = \mathbf{A}$, and \eqref{divergence scalar matrix} with $f = f(\cur)$ and $\boldsymbol{G} = \ds \X$,  we can derive
    \begin{subequations}
        \begin{align}
            \ds \cdot(\n (\ds f'(\cur))^T) &= \ds\n \ds f'(\cur) + \left(\lap f'(\cur)\right) \n \nonumber\\
            &= \mathbf{A} \ds f'(\cur) + \left(\lap f'(\cur)\right) \n,  \label{pf1 eq1, curvature-dependent}\\
            \ds \cdot (f'(\cur) \ds \n) &= \ds\n \ds f'(\cur)+ f'(\cur) \lap \n \nonumber\\
            &= \mathbf{A} \ds f'(\cur) + f'(\cur) \lap \n,  \label{pf1 eq2, curvature-dependent} \\
            \ds \cdot (f(\cur) \ds \X) &= f'(\cur) \ds \X \ds \cur + f(\cur) \lap \X \nonumber\\
            &= f'(\cur)  \ds \cur - \cur f(\cur) \n.   \label{pf1 eq3, curvature-dependent}
        \end{align}
    \end{subequations}
    Next, we consider the linear combination \eqref{pf1 eq1, curvature-dependent} $-$ \eqref{pf1 eq2, curvature-dependent} $+$ \eqref{pf1 eq3, curvature-dependent}, and invoke relation \eqref{gradient multi scalar} and \eqref{pf1 eq5}, we can get the following strong formulation of $V\n$:

    \begin{align}
        &\lap \left(f'(\cur)\n\right) - \ds \cdot \left(2\left(f'(\cur) \ds \n\right) - \left(f(\cur) \ds \X\right) \right) \nonumber\\
         &\ =  \ds \cdot \left(\left(\n (\ds f'(\cur))^T\right) - \left(f'(\cur) \ds \n\right) + \left(f(\cur) \ds \X\right) \right) \nonumber\\
         &\ = \left(\mathbf{A} \ds f'(\cur) + \left(\lap f'(\cur)\right) \n\right) + \left(f'(\cur)  \ds \cur - \cur f(\cur) \n\right)   \nonumber\\
        & \qquad - \left(\mathbf{A} \ds f'(\cur) + f'(\cur) \left(\ds \cur - |\mathbf{A}|^2 \n\right)\right) \nonumber\\
         &\ = \left(\lap f'(\cur)  + |\mathbf{A}|^2 f'(\cur) - \cur f(\cur) \right) \n \nonumber\\
        &\ = V\n.
    \end{align}
Finally, multiplying both sides of the above equation by $\w$, then integrating over $\Gamma(t)$, and 
applying the product rule \eqref{divergence matrix vector}, \eqref{gradient multi scalar}, and the divergence theorem  with $\boldsymbol{f} = \left(2\left(f'(\cur) \ds \n\right) - \left(f(\cur) \ds \X\right)\right)^T \boldsymbol{\omega}$, we have
\begin{align}
    \int_{\Gamma(t)}V\n \cdot \w\, \dd A=& \int_{\Gamma(t)}\left( \lap \left(f'(\cur)\n\right)- \ds\cdot \left(2\left(f'(\cur) \ds \n\right) - \left(f(\cur) \ds \X\right)\right) \right)\cdot \w\,\dd A\nonumber\\
    =&\int_{\Gamma(t)} \left(-\ds(f'(\cur)\n)+ \left(2 f'(\cur) \mathbf{A}- f(\cur) \ds\X\right)\right):\ds\w\,\dd A\nonumber\\
    &-\int_{\Gamma(t)} \cur\n^T(2f'(\cur) \mathbf{A}- f(\cur) \ds\X) \w\,\dd A.
\end{align}
\begin{align}
    \int_{\Gamma(t)}V\n \cdot \w\, \dd A=&\int_{\Gamma(t)} \left(-\n(\ds f'(\cur))^T-f'(\cur) \mathbf{A} \right):\ds\w\,\dd A\nonumber\\
    & \qquad + \int_{\Gamma(t)} \left(2 f'(\cur) \mathbf{A}- f(\cur) \ds\X\right): \ds \w\,\dd A \nonumber\\
    =&\int_{\Gamma(t)} \left(f'(\cur) \mathbf{A}  -\n (\ds f'(\cur))^T - f(\cur)\ds \X\right):\ds\w\,\dd A,
\end{align}
which yields the result \eqref{new identity for V, curvature-dependent geometric gradient flow}.\qed
\end{proof}

Next, we adopt the temporal evolution of $\cur$ derived in Lemma \ref{lemma H}, thereby deriving the following variational formulation for \eqref{eq: curvature-dependent geometric gradient flow}: Given the initial data $\X(\cdot, 0):= \X_0(\cdot)$ and $\cur(\cdot, 0):= -\n \cdot \lap \X_0(\cdot)$, find the solution $\X(\cdot,t)\in [H^1(\Gamma(t))]^3$, $V(\cdot,t)\in H^1(\Gamma(t))$, and $\cur(\cdot,t)\in H^1(\Gamma(t))$, such that
\begin{subequations}\label{new weak form, curvature-dependent geometric gradient flow}
\begin{align}
    &\left(\n \cdot \dt \X, \phi\right)_{\Gamma(t)} = \left(V, \phi\right)_{\Gamma(t)}, \quad \forall \phi \in H^1(\Gamma(t)),    
    \label{new weak form X, curvature-dependent geometric gradient flow}\\
    &\left(V\n, \w\right)_{\Gamma(t)} = \Big\langle f'(\cur) \mathbf{A} - \n(\ds f'(\cur))^T- f(\cur)\ds \X, \ds\w\Big\rangle_{\Gamma(t)},\nonumber \\
    & \qquad\qquad \hspace{3cm} \forall \w \in [H^1(\Gamma(t))]^3, \label{new weak form V, curvature-dependent geometric gradient flow}\\
    &\left(\dt\cur, \varphi\right)_{\Gamma(t)} = \big\langle\ds(\dt \X), \n (\ds \varphi)^T- \varphi  \mathbf{A} \big\rangle_{\Gamma(t)},\quad \forall   \varphi\in H^1(\Gamma(t)).\label{new weak form H, curvature-dependent geometric gradient flow}
\end{align}
\end{subequations}
If we take $f(\cur) = \frac{1}{2}\cur^2$, we observe that $f'(\cur) = \cur$, and this weak formulation \eqref{new weak form, curvature-dependent geometric gradient flow} goes to the weak formulation \eqref{new weak form} for the Willmore flow. 

For this weak formulation, we establish the following energy stability result.

\begin{theorem}[Energy dissipation]\label{thm:continuous energy dissipation, curvature-dependent geometric gradient flow}
    Let $\left(\X(\cdot,t), V(\cdot, t), \cur(\cdot,t)\right)$ be a sufficiently smooth solution of the new variational formulation \eqref{new weak form, curvature-dependent geometric gradient flow} with the initial data 
   $\X(\cdot,0)$ and $\cur(\cdot,0)$. 
   Then the curvature-dependent energy $W(\Gamma(t))$ defined in \eqref{eqn:continuous energy, curvature-dependent geometric gradient flow} is decreasing over time,
    \begin{equation}
    \label{eqn: continuous energy dissipation, curvature-dependent geometric gradient flow}
        W(t) \leq W(t')\leq W(0) = \int_{\Gamma_0} f(\cur_0) \,\dd A, \quad \forall t \geq t'\geq 0.
    \end{equation}
\end{theorem}

The proof follows completely analogous to that of the Willmore flow, with the only difference being that in the third equation we choose the test function $\varphi = f'(\cur)$. We therefore omit the detailed proof here for brevity.

\subsection{An energy-stable full discretization}
We then propose the following unconditionally energy-stable full discretization for the new variational formulation \eqref{new weak form, curvature-dependent geometric gradient flow}: Given the initial data of polygonal surface $\Gamma^0$: $\X^0 (\cdot) \in [\mathbb{K}^0]^3$  and $\cur^0(\cdot)\in \mathbb{K}^0$, for any $m\geq 0$, find the solution $\left(\X^{m+1},V^{m+1},\cur^{m+1}\right) \in [\mathbb{K}^m]^3 \times \mathbb{K}^m\times \mathbb{K}^m$, such that
\begin{subequations}
    \label{eqn:full curvature-dependent geometric gradient flow}
    \begin{align}
    \label{eqn:full1 curvature-dependent geometric gradient flow}
    &\left(\frac{\X^{m+1}-\X^m}{\tau}\cdot \n^m, \phi^h\right)_{\Gamma^m}^h = \left(V^{m+1}, \phi^h\right)_{\Gamma^m}^h, 
    \quad \forall \phi^h\in \mathbb{K}^m,\\
    \label{eqn:full2 curvature-dependent geometric gradient flow}
    &\left(V^{m+1}\n^m,\w^h\right)_{\Gamma^m}^h = \left\langle f'(\cur^{m+1})
    \mathbf{A}^{m}-{\n^m(\ds I_h(f'(\cur^{m+1})))^T}, \ds \w^h\right\rangle_{\Gamma^m}^h \nonumber\\
    &\qquad\qquad\qquad - \left\langle f(\cur^{m+1}) \ds \X^{m+1}, \ds \w^h\right\rangle_{\Gamma^m}^h,
    \quad \forall \w^h \in [\mathbb{K}^m]^3,\\
    \label{eqn:full3 curvature-dependent geometric gradient flow}
    &\left(\cur^{m+1}-\cur^m, \varphi^h\right)_{\Gamma^m}^h = \left\langle \ds(\X^{m+1}-\X^m), \n^m (\ds \varphi^h)^T \right\rangle_{\Gamma^m}^h \nonumber\\
    &\qquad\qquad\qquad \qquad \qquad - \left\langle \ds(\X^{m+1}-\X^m), \varphi^h \mathbf{A}^{m} \right\rangle_{\Gamma^m}^h,
    \ \ \forall \varphi^h \in \mathbb{K}^m;
    \end{align}
    \end{subequations}
where $ \mathbf{A}^{m}$ is the numerical approximation of extended Weingarten map $\ds\n$ given by \eqref{eqn:A^m}, {and $I_h$ is the standard interpolation operator defined in \cite[Definition 43(i)]{Barrett2020}}.

Define the discretized energy as
\begin{equation}
    \label{discretized curvature-dependent energy}
    W^m \coloneqq \left(f(\cur^m), 1\right)_{\Gamma^m}^h = \frac{1}{3}\sum_{j = 1}^J\sum_{k = 1}^3|\sigma_{j}^m| f\left(\cur^m\left(\q_{j_k}^m\right)\right), \qquad m\ge0.
\end{equation}
Then, we have the following energy stability for the proposed ES-PFEM \eqref{eqn:full curvature-dependent geometric gradient flow}.
\begin{theorem}[Unconditional energy dissipation of \eqref{eqn:full curvature-dependent geometric gradient flow}]{
Suppose that $f \in C^1(\mathbb{R})$ is nonnegative and weakly convex. Let $\left(\X^{m+1}, V^{m+1}, \cur^{m+1}\right)$ be a solution of the ES-PFEM \eqref{eqn:full curvature-dependent geometric gradient flow} starting from the initial data $\left(\X^0,\cur^0\right)$. Then the discretized curvature-dependent energy \eqref{discretized curvature-dependent energy} decreases for any $\tau>0$:
\begin{equation}\label{fulld stability, curvature-dependent geometric gradient flow}
    W^{m+1}\leq W^m \leq W^0 =  \left(f(\cur^0), 1\right)_{\Gamma^0}^h, \qquad m\ge0.
\end{equation}}
\end{theorem}
\begin{proof}
For any scalar functions $f$ and $g$, by utilizing the mass-lumped inner product,  we know that 
\begin{equation}\label{equiv1}
  (f, I_h(g))_{\Gamma^m}^h = (f,g)_{\Gamma^m}^h,
\end{equation}
%

    We follow the same proof strategy as in the Willmore flow. Let $\phi^h=-\tau V^{m+1} \in \mathbb{K}^m$ in \eqref{eqn:full1 curvature-dependent geometric gradient flow}, $\w=\X^{m+1}-\X^m \in [\mathbb{K}^m]^3$ in \eqref{eqn:full2 curvature-dependent geometric gradient flow}, and $\varphi^h = I_h(f'(\cur^{m+1})) \in \mathbb{K}^m$ in \eqref{eqn:full3 curvature-dependent geometric gradient flow}, we obtain
    \begin{subequations}
        \begin{align}
        \label{test func 1, curvature-dependent geometric gradient flow}
        &-\left((\X^{m+1}-\X^m)\cdot \n^m, V^{m+1}\right)_{\Gamma^m}^h = -\tau\left(V^{m+1},\, V^{m+1}\right)_{\Gamma^m}^h,\\[0.5em]
        \label{test func 2, curvature-dependent geometric gradient flow}
        &\left(V^{m+1}\n^m,\X^{m+1}-\X^m\right)_{\Gamma^m}^h = \Big\langle f'(\cur^{m+1})) \mathbf{A}^{m},\, \ds (\X^{m+1}-\X^m)\Big\rangle_{\Gamma^m}^h\nonumber\\
        &  \qquad\qquad\qquad\qquad\qquad\quad\quad-\Big\langle {\n^m(\ds I_h(f'(\cur^{m+1})))^T},\, \ds (\X^{m+1}-\X^m)\Big\rangle_{\Gamma^m}^h\nonumber \\
        &\qquad\qquad\qquad\qquad\qquad\quad\quad-\Big\langle f(\cur^{m+1}) \ds \X^{m+1},\, \ds (\X^{m+1}-\X^m)\Big\rangle_{\Gamma^m}^h,\\[0.5em]
        \label{test func 3, curvature-dependent geometric gradient flow}
        &\left(\cur^{m+1}-\cur^m, I_h(f'(\cur^{m+1}))\right)_{\Gamma^m}^h = \Big\langle \ds(\X^{m+1}-\X^m), \n^m (\ds I_h(f'(\cur^{m+1})))^T  \Big\rangle_{\Gamma^m}^h \nonumber \\
        &\qquad\qquad\qquad\qquad\qquad - \Big\langle \ds(\X^{m+1}-\X^m), I_h(f'(\cur^{m+1}))) \mathbf{A}^{m} \Big\rangle_{\Gamma^m}^h.
        \end{align}
        \end{subequations}
   {
By \eqref{equiv1}, \eqref{test func 3, curvature-dependent geometric gradient flow}
is equivalent to
\begin{align}\label{equiv, test func 3, curvature-dependent gradient flow}
    & \left(\cur^{m+1}-\cur^m, f'(\cur^{m+1})\right)_{\Gamma^m}^h
    =
    \Big\langle
    \ds(\X^{m+1}-\X^m),
    \n^m \big(\ds I_h(f'(\cur^{m+1}))\big)^T
    \Big\rangle_{\Gamma^m}^h
    \nonumber \\
    &\qquad\qquad\qquad\qquad\qquad\qquad\quad
    -
    \Big\langle
    \ds(\X^{m+1}-\X^m),
    f'(\cur^{m+1}) \mathbf{A}^{m}
    \Big\rangle_{\Gamma^m}^h .
\end{align}
}
    Summing \eqref{test func 1, curvature-dependent geometric gradient flow}, \eqref{test func 2, curvature-dependent geometric gradient flow} and \eqref{equiv, test func 3, curvature-dependent gradient flow}, we derive that
    \begin{align}\label{aeq 0, curvature-dependent geometric gradient flow}
        &\left(\cur^{m+1}-\cur^m, \,f'(\cur^{m+1})\right)_{\Gamma^m}^h + \Big\langle f(\cur^{m+1})\ds \X^{m+1},\, \ds\X^{m+1}-\ds \X^m\Big\rangle_{\Gamma^m}^h \nonumber\\
        &=-\tau\left(V^{m+1},\, V^{m+1}\right)_{\Gamma^m}^h \leq 0.
    \end{align}  
    By the convexity and nonnegativity of $f$, we know that $f(\cur^{m})\geq f(\cur^{m+1}) + f'(\cur^{m+1})(\cur^{m}-\cur^{m+1})$. Therefore, we have
    \begin{align}\label{aeq 1, curvature-dependent geometric gradient flow}
    \left(\cur^{m+1}-\cur^m, f'(\cur^{m+1})\right)_{\Gamma^m}^h&\geq \left(f(\cur^{m+1}) - f(\cur^m), 1\right)_{\Gamma^m}^h\nonumber\\[0.5em]
    &= \left(f(\cur^{m+1}) , 1\right)_{\Gamma^m}^h-W^m, \qquad m\ge0.
\end{align}
    Using the same argument as in the Willmore flow, we can derive that
    \begin{align}\label{aeq 2, curvature-dependent geometric gradient flow}
        \Big\langle f(\cur^{m+1})\ds \X^{m+1},\, \ds\X^{m+1}-\ds \X^m\Big\rangle_{\Gamma^m}^h \geq W^{m+1} - \left(f(\cur^{m+1}) , 1\right)_{\Gamma^m}^h.
    \end{align}
    All together, we have
    \begin{align}\label{aeq 3, curvature-dependent geometric gradient flow}
        W^{m+1} - W^m \leq -\tau \left(V^{m+1},\, V^{m+1}\right)_{\Gamma^m}^h \leq 0, \qquad m\ge0,
    \end{align}
    which implies the energy dissipation we desired.\qed
\end{proof}

\subsection{Application to the Gauss curvature flow}
We now apply the ES-PFEM \eqref{eqn:full curvature-dependent geometric gradient flow} to the special case of the Gauss curvature flow, i.e., $f(\cur) = \cur$. The Gauss curvature flow is of particular interest for two reasons. First, the linear density function $f(\cur) = \cur$ results in a weakly nonlinear numerical scheme. Second, despite its importance in geometric evolution, there is no prior PFEM proposed for the Gauss curvature flow for surfaces.

For the Gauss curvature flow with $f(\cur) = \cur$, we have $f'(\cur) = 1$, $\nabla_\Gamma f'(\cur) = 0$, and the general scheme \eqref{eqn:full curvature-dependent geometric gradient flow} simplifies to: Given the initial data $\X^0 (\cdot) \in [\mathbb{K}^0]^3$ and $\cur^0(\cdot)\in \mathbb{K}^0$, for any $m\geq 0$, find $\left(\X^{m+1},V^{m+1},\cur^{m+1}\right) \in [\mathbb{K}^m]^3 \times \mathbb{K}^m\times \mathbb{K}^m$ such that
\begin{subequations}
    \label{eqn:full Gauss curvature flow}
    \begin{align}
    \label{eqn:full1 Gauss curvature flow}
    &\left(\frac{\X^{m+1}-\X^m}{\tau}\cdot \n^m, \phi^h\right)_{\Gamma^m}^h = \left(V^{m+1}, \phi^h\right)_{\Gamma^m}^h, 
    \quad \forall \phi^h\in \mathbb{K}^m,\\
    \label{eqn:full2 Gauss curvature flow}
    &\left(V^{m+1}\n^m,\w^h\right)_{\Gamma^m}^h = \left\langle \mathbf{A}^{m}, \ds \w^h\right\rangle_{\Gamma^m}^h 
    - \left\langle \cur^{m+1} \ds \X^{m+1}, \ds \w^h\right\rangle_{\Gamma^m}^h,
    \quad \forall \w^h \in [\mathbb{K}^m]^3,\\
    \label{eqn:full3 Gauss curvature flow}
    &\left(\cur^{m+1}-\cur^m, \varphi^h\right)_{\Gamma^m}^h = \left\langle \ds(\X^{m+1}-\X^m), \n^m (\ds \varphi^h)^T \right\rangle_{\Gamma^m}^h \nonumber\\
    &\qquad\qquad\qquad\qquad - \left\langle \ds(\X^{m+1}-\X^m), \varphi^h \mathbf{A}^{m} \right\rangle_{\Gamma^m}^h,
    \ \ \forall \varphi^h \in \mathbb{K}^m.
    \end{align}
    \end{subequations}
Notably, the weak nonlinearity appears in only one term $\left\langle \cur^{m+1} \ds \X^{m+1}, \ds \w^h\right\rangle_{\Gamma^m}^h$, which in general comes from \eqref{eqn:full2 curvature-dependent geometric gradient flow}: 
\begin{equation*}
    \left\langle f(\cur^{m+1}) \ds \X^{m+1}, \ds \w^h\right\rangle_{\Gamma^m}^h.
\end{equation*}
This nonlinearity stems naturally from the energy density $f(\cur)$ itself. Moreover, it is essential in the energy stability proofs for analyzing energy differences between surfaces, i.e.,
\begin{equation*}
    \Big\langle f(\cur^{m+1})\ds \X^{m+1},\, \ds(\X^{m+1}- \X^m)\Big\rangle_{\Gamma^m}^h \geq \left( f(\cur^{m+1}) , 1\right)_{\Gamma^{m+1}}^h - \left( f(\cur^{m+1}) , 1\right)_{\Gamma^m}^h.
\end{equation*}


\begin{remark}
    If we take $f(\cur) \equiv 1$, then $f'(\cur) \equiv 0$, and the weak formulation \eqref{new weak form, curvature-dependent geometric gradient flow} reduces to the following:
    \begin{subequations}\label{new weak form, MCF}
        \begin{align}
            &\left(\frac{\X^{m+1}-\X^m}{\tau}\cdot \n^m, \phi^h\right)_{\Gamma^m}^h = \left(V^{m+1}, \phi^h\right)_{\Gamma^m}^h, 
            \quad \forall \phi^h\in \mathbb{K}^m,\label{new weak form X, MCF}\\
            &\left(V^{m+1}\n^m,\w^h\right)_{\Gamma^m}^h = 
            - \left\langle \ds \X^{m+1}, \ds \w^h\right\rangle_{\Gamma^m}^h,
            \quad \forall \w^h \in [\mathbb{K}^m]^3,\label{new weak form V, MCF}\\
            &\left(\cur^{m+1}-\cur^m, \varphi^h\right)_{\Gamma^m}^h = \left\langle \ds(\X^{m+1}-\X^m), \n^m (\ds \varphi^h)^T \right\rangle_{\Gamma^m}^h \nonumber\\
    &\qquad\qquad\qquad\qquad - \left\langle \ds(\X^{m+1}-\X^m), \varphi^h \mathbf{A}^{m} \right\rangle_{\Gamma^m}^h,
    \ \ \forall \varphi^h \in \mathbb{K}^m.
        \end{align}
    \end{subequations}
    Note that the evolution equation \eqref{new weak form H, curvature-dependent geometric gradient flow} becomes redundant since the first two equations \eqref{new weak form X, MCF}-\eqref{new weak form V, MCF} do not involve the mean curvature $\cur$. Therefore, the curvature evolution equation \eqref{new weak form H, curvature-dependent geometric gradient flow} is no longer needed, and the system reduces to the first two equations. Remarkably, the first two equations \eqref{new weak form, MCF} coincides exactly with the BGN weak formulation for the mean curvature flow \cite{barrett2008parametric1}:
    \begin{subequations}
        \begin{align*}
            &\left(\frac{\X^{m+1}-\X^m}{\tau}\cdot \n^m, \phi^h\right)_{\Gamma^m}^h = -\left(\cur^{m+1}, \phi^h\right)_{\Gamma^m}^h, 
            \quad \forall \phi^h\in \mathbb{K}^m, \\
            &\left(\cur^{m+1}\n^m,\w^h\right)_{\Gamma^m}^h = 
            \left\langle \ds \X^{m+1}, \ds \w^h\right\rangle_{\Gamma^m}^h,
            \quad \forall \w^h \in [\mathbb{K}^m]^3,
        \end{align*}
    \end{subequations}
    with $V^{m+1}$ replaced by $-\cur^{m+1}$. This reveals an interesting connection: despite that our formulation \eqref{new weak form, curvature-dependent geometric gradient flow} comprises three equations while the BGN formulation contains only two, our formulation can still be viewed as a generalization of the BGN formulation.
\end{remark}

\begin{remark}
    By slightly modifying the evolution equation for $\cur$, we can obtain another energy-stable scheme, which can directly reduced to the BGN formulation in \cite{barrett2008parametric1}. Instead of evolving $\cur$, we consider the evolution equation for $f(\cur)$:
    \begin{equation*}
        \label{new identity for H, curvature-dependent geometric gradient flow}
        \int_{\Gamma(t)} \dt \cur (f'(\cur) \varphi) \,\dd A = \int_{\Gamma(t)}\ds (\dt\X): \left(\n(\ds(f'(\cur) \varphi))^T- f'(\cur) \varphi \mathbf{A}\right)\,\dd A.
    \end{equation*}
    Then we derive the following weak formulation
    \begin{subequations}\label{new weak form, curvature-dependent geometric gradient flow, 2}
        \begin{align}
            &\left(\n \cdot \dt \X, \phi\right)_{\Gamma(t)} = \left(V, \phi\right)_{\Gamma(t)}, \quad \phi \in H^1(\Gamma(t)),    
            \label{new weak form X, curvature-dependent geometric gradient flow, 2}\\
            &\left(V\n, \w\right)_{\Gamma(t)} = \Big\langle f'(\cur) \mathbf{A} - \n(\ds f'(\cur))^T- f(\cur)\ds \X, \ds\w\Big\rangle_{\Gamma(t)},\nonumber \\
            & \qquad\qquad \qquad \qquad \qquad \w \in [H^1(\Gamma(t))]^3, \label{new weak form V, curvature-dependent geometric gradient flow, 2}\\
            &\left(\dt\cur, f'(\cur) \varphi\right)_{\Gamma(t)} = \big\langle\ds(\dt \X), \n (\ds(f'(\cur) \varphi))^T- \varphi f'(\cur) \mathbf{A} \big\rangle_{\Gamma(t)},\nonumber \\
            & \qquad \qquad\qquad \qquad \qquad  \varphi \in H^1(\Gamma(t)).\label{new weak form H, curvature-dependent geometric gradient flow, 2}
        \end{align}
    \end{subequations}
    When $f(\cur) = 1$, the evolution equation \eqref{new weak form H, curvature-dependent geometric gradient flow, 2} vanishes since it becomes the trivial identity $0 = 0$. This reveals an interesting connection: despite that our formulation \eqref{new weak form, curvature-dependent geometric gradient flow, 2} comprises three equations while the BGN formulation contains only two, our formulation can still be viewed as a generalization of the BGN formulation.
\end{remark}

\section{Mesh improvement and Newton iterations}
In the end of the previous section, we build upon the energy-stable PFEM \eqref{new weak form, curvature-dependent geometric gradient flow, 2}, which reduces to the BGN scheme when $f(\cur)\equiv 1$. This reduction suggests that our formulation may inherit the good mesh quality of the BGN scheme. But in practice, we have observed that the ES-PFEM \eqref{eqn:full wm} in Section 2 often does not exhibit favorable mesh quality.

In fact, the BGN scheme's mesh quality is improved by the mean curvature vector identity  $\cur\n=-\lap\X$, which simultaneously provides the computational formula for $\cur$. In contrast, for the general energy density $f(\cur)$, the computation of mean curvature $\cur$ and mesh quality are governed by distinct equations: the evolution equation for $\cur$ in \eqref{new weak form H, curvature-dependent geometric gradient flow} and the normal velocity equation for $V\n$ in \eqref{new weak form V, curvature-dependent geometric gradient flow}, respectively. This decoupling allows us to enhance mesh quality by adopting the tangential motion control methodology from \cite{barrett2008parametric2} to modify the $V\n$ equation, while preserving the $\cur$ computation to maintain energy stability. Using the ES-PFEM \eqref{eqn:full wm} as a representative example, we develop an ES-PFEM for the Willmore flow with improved mesh quality, and propose Newton iterative solvers for the resulting nonlinear system.

\subsection{An improved ES-PFEM with controlled tangential motion}
Even though we entirely employ purely normal motion in the derivation of the continuous variational formulation \eqref{new weak form} and the BGN scheme is implicitly included in \eqref{eqn:full curvature-dependent geometric gradient flow}, our fully discretized scheme still improperly introduces significant implicit tangential component which leads to bad mesh quality. The mesh points under the ES-PFEM \eqref{eqn:full wm} will easily cluster or disperse within a short period, resulting in extremely low robustness of the algorithm.

To control the tangential motion during evolution and therefore improve the mesh quality, the crucial idea is to represent the tangential part of velocity explicitly and then control this part by a given parameter \cite{barrett2008parametric2}.  In the continuous level, reformulate the variational formulation \eqref{new weak form} as follows
\begin{subequations}\label{new weak form reduced}
\begin{align}
    &\left(\n \cdot \dt \X, \phi\right)_{\Gamma(t)} = \left(V, \phi\right)_{\Gamma(t)}, \quad \forall \phi\in L^2(\Gamma(t)),\label{new weak form X reduced}\\
    &\left(\Tau_i \cdot \dt \X, \psi_i\right)_{\Gamma(t)} = \left(\beta_i, \psi_i\right)_{\Gamma(t)}, \quad i = 1,2, \quad \forall \psi_1, \psi_2\in L^2(\Gamma(t)),\label{new weak form b reduced}\\
    &\left(V\n + \alpha \sum_{i = 1,2} \beta_i\Tau_i, \w\right)_{\Gamma(t)} = \Big \langle\cur \mathbf{A}- \n\left(\ds\cur\right)^T- \frac{1}{2}\cur^2\ds \X, \ds\w\Big\rangle_{\Gamma(t)},\nonumber\\
    &\qquad \qquad \qquad \qquad \qquad  \forall \w\in [H^1(\Gamma(t))]^3,\label{new weak form V reduced}\\
    &\left(\dt\cur, \varphi\right)_{\Gamma(t)} = \Big\langle\ds(\dt \X), \n (\ds \varphi)^T- \varphi \mathbf{A} \Big\rangle_{\Gamma(t)}, \quad \forall \varphi \in H^1(\Gamma(t)),\label{new weak form H reduced}
\end{align}
\end{subequations}
where $\Tau_1, \Tau_2$ are two mutually non-parallel tangent vector fields and $\text{span}\{\n, \Tau_1, \Tau_2\} = \bR^3$, {
$\alpha>0$ is a hyperparameter for mesh-quality control, which will be specified
in the fully discrete scheme below. The quantities $\beta_i$, $i=1,2$, represent
the tangential velocity components along $\Tau_i$. In the fully discrete scheme,
these components will be penalized by $\alpha$, hence a larger $\alpha$ enforces a
stronger suppression of tangential motion.}
{
\begin{remark}
  The shape evolution of a surface is governed by its normal velocity. Therefore, the weak forms \eqref{new weak form} and \eqref{new weak form reduced}, which have the same normal motion, describe the same geometric shape evolution of the surface.
\end{remark}}

 
 Thus, the fully discretized scheme based on the improved formulation \eqref{new weak form reduced} is given as follows: Given the initial data of polygonal surface $\Gamma^0$: $\X^0 (\cdot) \in [\mathbb{K}^m]^3$  and $\cur^0(\cdot)\in \mathbb{K}^m$ , for any $m\geq 0$, find the solution $\left(\X^{m+1},V^{m+1},\beta^{m+1}_1, \beta^{m+1}_2, \cur^{m+1}\right) \in [\mathbb{K}^m]^3 \times \mathbb{K}^m\times\mathbb{K}^m\times\mathbb{K}^m\times \mathbb{K}^m$, such that{
\begin{subequations}
\label{eqn:full wm reduced}
\begin{align}
\label{eqn:full1 wm reduced}
&\left(\frac{\X^{m+1}-\X^m}{\tau}\cdot \n^m, \phi^h\right)_{\Gamma^m}^h= \left(V^{m+1},\,  \phi^h\right)_{\Gamma^m}^h,\qquad \qquad \quad  \forall \phi^h\in \mathbb{K}^m,\\[0.5em]
\label{eqn:full2 wm reduced}
&\left(\frac{\X^{m+1}-\X^m}{\tau}\cdot \Tau^m_i, \psi^h_i\right)_{\Gamma^m}^h= \left(\beta_i^{m+1},\,  \psi^h_i\right)_{\Gamma^m}^h,\qquad  i = 1,2, \quad\forall \psi^h_1, \psi^h_2\in \mathbb{K}^m,\\[0.5em]
\label{eqn:full3 wm reduced}
&\left(V^{m+1}\n^m+ \alpha^m\sum_{i = 1,2} \beta^{m+1}_i \Tau_i^m ,\w^h\right)_{\Gamma^m}^h=\Big\langle\cur^{m+1} \mathbf{A}^m-\n^m(\ds \cur^{m+1})^T,\, \ds \w^h\Big\rangle_{\Gamma^m}^h\nonumber\\
& \qquad \qquad  \qquad\qquad-\Big\langle\frac{1}{2}(\cur^{m+1})^2 \ds \X^{m+1},\, \ds \w^h\Big\rangle_{\Gamma^m}^h,\quad \forall \w^h  \in [\mathbb{K}^m]^3, \\[0.5em]
\label{eqn:full4 wm reduced}
&\left(\cur^{m+1}-\cur^m, \,\varphi^h\right)_{\Gamma^m}^h = \Big\langle \ds(\X^{m+1}-\X^m), \n^m (\ds\varphi^h)^T -\varphi^h \mathbf{A}^m \Big\rangle_{\Gamma^m}^h,\quad  \forall \varphi^h \in \mathbb{K}^m;
\end{align}
\end{subequations}
where $\alpha^m$ is a hyperparameter. Here the two non-parallel unit tangential vectors $\Tau_1^m$ and $\Tau_2^m$ of the polynomial surface are defined as follows: on each surface element $\sigma_j^m = \{\q_{j_1}^m, \q_{j_2}^m, \q_{j_3}^m\}$

\begin{equation}
\boldsymbol{\tau}^m_1\big|_{\sigma_j^m}\coloneqq  \frac{\q_{j_2}^m- \q_{j_1}^m}{|\q_{j_2}^m- \q_{j_1}^m|}, \quad \boldsymbol{\tau}^m_2\big|_{\sigma_j^m}\coloneqq  \boldsymbol{\tau}^m_1\big|_{\sigma_j^m} \times \n^m\big|_{\sigma_j^m}. 
\end{equation}

}



Notably, the ES-PFEM \eqref{eqn:full wm reduced} preserves energy dissipation.
\begin{theorem}[Unconditional energy dissipation of \eqref{eqn:full wm reduced}]\label{thm: energy2}
    Let $\left(\X^{m+1}, V^{m+1} ,\cur^{m+1}\right)$ be a solution of our ES-PFEM \eqref{eqn:full wm reduced} with initial data $\left(\X^0,\cur^0\right)$. Then the discretized Willmore energy \eqref{discretized Willmore energy} decreases with respect to $m$ unconditionally when the hyperparameter $\alpha^m\geq 0$.
\end{theorem}
\begin{proof}
    Taking $\phi^h = -\tau V^{m+1}\in \mathbb{K}^m$ in \eqref{eqn:full1 wm reduced}, $\psi^h_i = -\alpha^m\beta^{m+1}_i\in \mathbb{K}^m$ in \eqref{eqn:full2 wm reduced} for $i = 1,2$, $\w^h = \X^{m+1}-\X^m\in [\mathbb{K}^m]^3$ in \eqref{eqn:full3 wm reduced} and $\varphi^h = \cur^{m+1} \in \mathbb{K}^m$ in \eqref{eqn:full4 wm reduced}, and summing them together, we derive that{
    \begin{align}\label{pf6 eq1}
        &\left(\cur^{m+1}-\cur^m, \,\cur^{m+1}\right)_{\Gamma^m}^h + \Big\langle\frac{1}{2}(\cur^{m+1})^2\ds \X^{m+1},\, \ds\X^{m+1}-\ds \X^m\Big\rangle_{\Gamma^m}^h \nonumber\\
        &=-\tau\left(V^{m+1},\, V^{m+1}\right)_{\Gamma^m}^h -\tau\sum_{i = 1,2}\left(\beta^{m+1}_i,\beta_i^{m+1}\right)_{\Gamma^m}^h\leq 0,
    \end{align}}
    which is similar to the equation \eqref{aeq 0}. Therefore, the subsequent proof follows identically to \eqref{aeq 5}, which implies the energy dissipation we desired.\qed
\end{proof}

\subsection{Newton iterative solver}
To solve the fully implicit numerical scheme \eqref{eqn:full wm reduced}, we employ the Newton-Rhapson iterations.

Given $\left(\X^m, \cur^m\right)\in [\mathbb{K}^m]^3 \times \mathbb{K}^m$, at time step $t_{m+1}$, set $\X^{m+1,0}=\X^m$ and $\cur^{m+1,0}=\cur^m$ as the initial guess of the nonlinear problem \eqref{eqn:full wm reduced}, where $V^{m+1,0}$ and $\beta_i^{m+1,0}$ are chosen arbitrarily because we do not need their values at $t = t_m$ in the scheme \eqref{eqn:full wm reduced}. And in the $l$-th iteration, given $\left(\boldsymbol{X}^{m+1, l}, V^{m+1,l}, \beta^{m+1,l}_1, \beta^{m+1,l}_2, \cur^{m+1,l}\right) \in [\mathbb{K}^m]^3 \times \mathbb{K}^m\times \mathbb{K}^m\times \mathbb{K}^m\times \mathbb{K}^m $ , the Newton direction $\left(\boldsymbol{X}^{\delta}, V^{\delta}, \beta^{\delta}_1, \beta^{\delta}_2, \cur^{\delta}\right)$  can be obtained by solving
\begin{subequations}
    \label{iteration}
    \begin{align}
        &\left(\frac{1}{\tau}\X^{\delta}\cdot \n^m, \phi^h\right)_{\Gamma^m}^h -\left(V^{\delta}, \phi^h\right)_{\Gamma^m}^h= -\left(\frac{\X^{m+1,l}-\X^m}{\tau}\cdot \n^m, \phi^h\right)_{\Gamma^m}^h
        +\left(V^{m+1,l},\phi^h\right)_{\Gamma^m}^h,\\[0.5em]
         &{\left(\frac{1}{\tau}\X^{\delta}\cdot \Tau_i^m, \psi^h_i\right)_{\Gamma^m}^h -\left(\beta_i^{\delta}, \phi^h\right)_{\Gamma^m}^h= -\left(\frac{\X^{m+1,l}-\X^m}{\tau}\cdot \Tau_i^m, \phi^h\right)_{\Gamma^m}^h
        +\left(\beta_i^{m+1,l},\phi^h\right)_{\Gamma^m}^h,\nonumber}\\
       & \qquad \qquad \qquad \qquad \qquad \qquad \qquad  \qquad \qquad \qquad  \qquad \qquad \qquad \quad i = 1,2,\\[0.5em]
         &\Big\langle\frac{1}{2}(\cur^{m+1,l})^2\ds\X^{\delta}, \ds\w^h\Big \rangle_{\Gamma^m}^h + \left(V^{\delta}\n^m + \alpha^m \sum_{i = 1,2}\beta_i^{\delta}\Tau_i^m, \w^h\right)_{\Gamma^m}^h\nonumber\\ 
         & \qquad -\Big\langle\cur^{\delta}\mathbf{A}^m- \n^m(\ds\cur^{\delta})^T, \ds\w^h\Big\rangle_{\Gamma^m}^h+\Big\langle\cur^{m+1,l}\cur^{\delta}\ds\X^{m+1,l}, \ds\w^h\Big \rangle_{\Gamma^m}^h\nonumber\\
        & \quad = -\left(V^{m+1,l}\n^m + \alpha^m \sum_{i= 1,2}\beta_i^{m+1,l}\Tau_i^m ,\w^h\right)_{\Gamma^m}^h-\Big\langle\frac{1}{2}(\cur^{m+1,l})^2 \ds\X^{m+1,l}, \ds\w^h\Big\rangle_{\Gamma^m}^h\nonumber\\
         & \qquad +\Big\langle\cur^{m+1,l}\mathbf{A}^m-\n^m (\ds\cur^{m+1,l})^T, \ds \w^h\Big\rangle_{\Gamma^m}^h, \\[0.5em]
         &- \Big\langle\ds\X^{\delta}, \n^m(\ds\varphi^h)^T- \varphi^h\mathbf{A}^m\Big \rangle_{\Gamma^m}^h+\left(\cur^{\delta},\varphi^h\right)_{\Gamma^m}^h = -\left(\cur^{m+1,l}-\cur^{m},\varphi^h\right)_{\Gamma^m}^h\nonumber\\
         &  \qquad \qquad \qquad \qquad\qquad  +\Big\langle\ds(\X^{m+1,l}-\X^m), \n^m(\ds\varphi^h)^T- \varphi^h\mathbf{A}^m\Big \rangle_{\Gamma^m}^h;
    \end{align}
\end{subequations}
for all $\phi^h \in \mathbb{K}^m$, $\psi^h_1\in \mathbb{K}^m$, $\psi_2^h \in \mathbb{K}^m$, $\boldsymbol{\omega}^h \in [\mathbb{K}^m]^3$ and $\varphi^h \in \mathbb{K}^m$. And for each $m\geq 0$, update the solution by 
$$\X^{m+1,l+1} = \X^{m+1, l} +\X^{\delta},\quad V^{m+1, l+1} = V^{m+1,l}+ V^{\delta},$$
$$\beta_i^{m+1,l+1} = \beta^{m+1,l}_i + \beta_i^{\delta},\quad \cur^{m+1,l+1} = \cur^{m+1,l}+ \cur^{\delta}.$$
 The iteration terminates when
$$
\max\left\{ \Vert \boldsymbol{X}^{\delta}\Vert_{L^{\infty}}, \Vert V^{\delta}\Vert_{L^{\infty}}, \Vert \beta_i^{\delta}\Vert_{L^{\infty}}, \Vert \cur^{\delta}\Vert_{L^{\infty}}\right\}\leq \text{tol},
$$
where $\Vert\cdot \Vert_{L^{\infty}}$ is defined as 
\begin{equation}
    \Vert f^h\Vert_{L^{\infty}} = \max_{1\leq k\leq K}|f^h(\q_k^m)|, \quad f^h \in \mathbb{K}^m, 
\end{equation}
and tol is the chosen tolerance to ensure that the iteration achieves a certain level of accuracy when the conditions are met.  The numerical solution of (\ref{eqn:full wm reduced}) is given by
$$
\boldsymbol{X}^{m+1} =\boldsymbol{X}^{m+1,l+1}, \quad V^{m+1} = V^{m+1,l+1},\quad \beta_i^{m+1} = \beta_i^{m+1,l+1} ,\quad \cur^{m+1} = \cur^{m+1,l+1}. 
$$

After the Newton iterations, the hyperparameter $\alpha^m$ is adjusted to $\alpha^{m+1}$ by the following algorithm according to $\beta_i^{m+1}$. When the $\beta_1^m$ and $\beta_2^m$ are too large/small, we increase/decrease the value of the hyperparameter by a factor $C$.
\begin{algorithm}\label{alg:1}
\caption{Adaptive $\alpha^m$}
\begin{algorithmic}[1]
\Require $\alpha^{m}$: hyperparameter at $t_m$; \  $\beta_i^{m+1}$: results from the Newton iterations;\  $C$: a constant $C>1$

\If{$\beta_i^{m+1} \geq 10^{-3}$}
    \State $\alpha^{m+1}= C\alpha^m$;
\ElsIf{$\beta_i^{m+1} \leq 10^{-6}$}
    \State $\alpha^{m+1}= \alpha^m/C$;
\EndIf

\end{algorithmic}
\end{algorithm}

To initialize the iterative solver, we also need the initial data $\cur^0$ and the initial hyperparameter $\alpha^0$. The curvature $\cur^0$ is approximated via the definition of $\cur$ \cite{barrett2008parametric2},
\begin{equation}
    \left(\cur^0, \chi^h\right)_{\Gamma^0}^h = \left(\text{Tr}(\mathbf{A}^0),\chi^h\right)_{\Gamma^0}^h, \quad \forall \chi^h \in \mathbb{K}^0.
\end{equation}
And $\alpha^0$ is subjectively determined in our subsequent numerical tests.

\begin{remark}
    We must emphasize that the method employed in this section inherently aims to suppress tangential motion. While theoretically, this suppression could be achieved through discretizing purely normal motion, but our actual code implementation revealed that the numerical scheme derived from purely normal motion  still leads to rapid mesh collapse.
\end{remark}

{
\begin{remark}
    The above mesh improvement approach and Newton iteration can be similarly applied to curvature-dependent geometric gradient flows. We omit the details for brevity. Throughout the following numerical tests, we use the improved scheme \eqref{eqn:full wm reduced} unless the original scheme  \eqref{eqn:full wm} is specially mentioned. Some examples will be presented below to illustrate the necessity of the improved scheme.
\end{remark}}

\section{Numerical results}

In this section, numerical experiments are carried out to validate the performance of our proposed ES-PFEM \eqref{eqn:full wm reduced}, which includes the investigation of convergence order, the validation of energy dissipation and morphological evolutions of closed surfaces.

Here are several basic initial surfaces we consider in our tests:
\begin{itemize}
\item Unit sphere: $x^2 +y^2 + z^2 = 1$.
\item Ellipsoid: $\frac{x^2}{a} + \frac{y^2}{b} + z^2 = 1$.
\item Torus: $(R- \sqrt{x^2+y^2})^2 + z^2 = r^2$.
\end{itemize}
And among all of our numerical tests, tol in the Newton iterations is chosen as $10^{-10}$. For adaptive $\alpha^m$,  the initial hyperparameter  $\alpha^0 = 10^{3}$ and constant $C = 5$ in Algorithm 1.  

{ In our numerical experiments, the improved method \eqref{eqn:full wm reduced} plays a crucial role in preserving mesh quality. As shown in Fig.~\ref{fig: compared to original scheme}, the mesh deteriorates rapidly when the original ES-PFEM \eqref{eqn:full wm} is employed. For this reason, all numerical experiments presented below are conducted using the improved ES-PFEM \eqref{eqn:full wm reduced}.}
\begin{figure}[htp!]
    \centering
\includegraphics[width=1\textwidth]{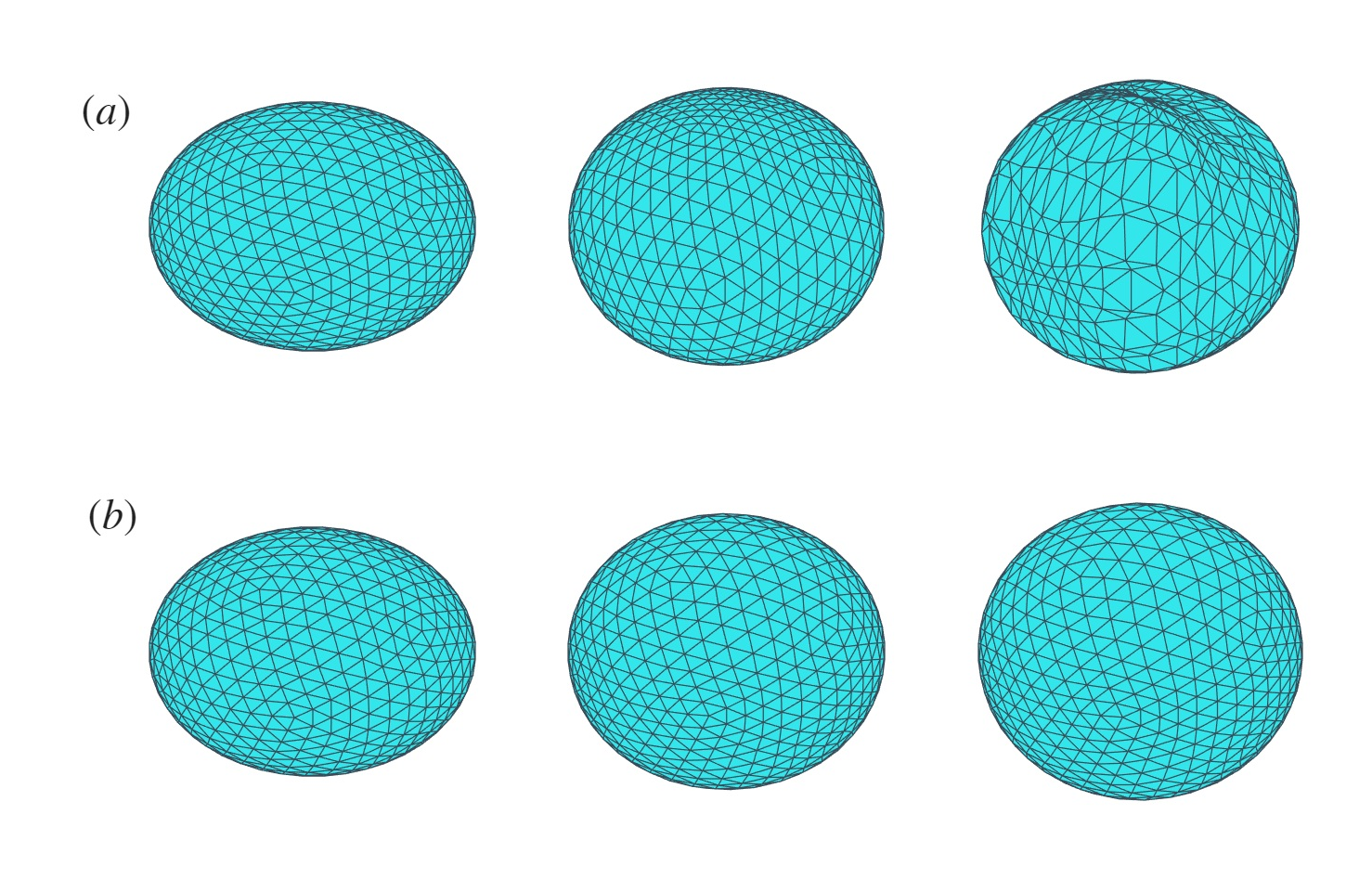}    
    \caption{Evolution of an ellipsoid with $a = b
     = \sqrt{2}$ at $ t = 0, 0.08, 0.14$. (a) Using the original ES-PFEM \eqref{eqn:full wm},
(b) Using the improved ES-PFEM \eqref{eqn:full wm reduced}.} 
    \label{fig: compared to original scheme}
\end{figure}
\subsection{Numerical tests for the Willmore flow}

The sphere and Clifford torus (torus with $R:r = \sqrt{2}:1$)  are critical points of the Willmore energy and thus are the steady-state solutions of the Willmore flow. Specifically, spheres globally minimize the Willmore energy among all closed surfaces with a value of $8\pi$\cite{willmore1993riemannian}. Among all closed surfaces with positive genus, the minimal Willmore energy is proved to be $4\pi^2$ achieved by the Clifford torus \cite{marques2014min,marques2014willmore}. 

To perform a convergence test for ES-PFEM \eqref{eqn:full wm reduced}, we define the mesh size $h$ of the initial polygonal surface $\Gamma^0 = \bigcup_{j =1}^J\overline{\sigma_j^0}$ as $h \coloneqq \max_{j = 1}^J \sqrt{|\sigma_j^0|}$. And we choose the time step size $\tau$ and the mesh size $h$ to satisfy $\tau = 0.01h^2$. This choice is based on the first-order temporal discretization and second-order spatial discretization in our scheme, thus $\tau = \mathcal{O}(h^2)$ can be chosen as the Cauchy path.
\begin{figure}[htp!]
\centering
\includegraphics[width=0.48\textwidth]{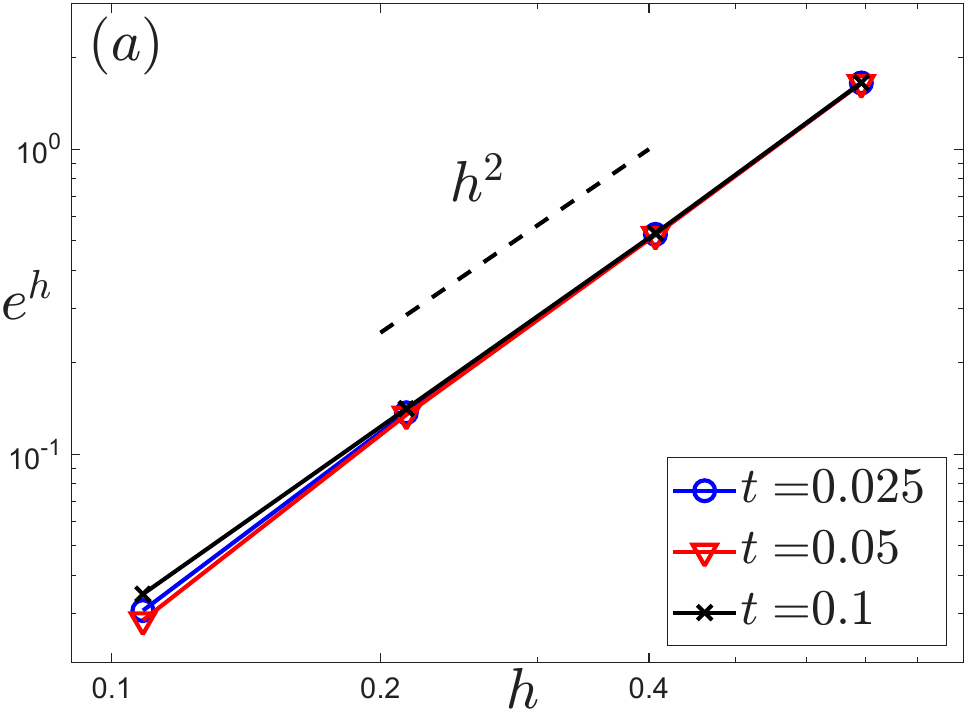}\quad \includegraphics[width=0.48\textwidth]{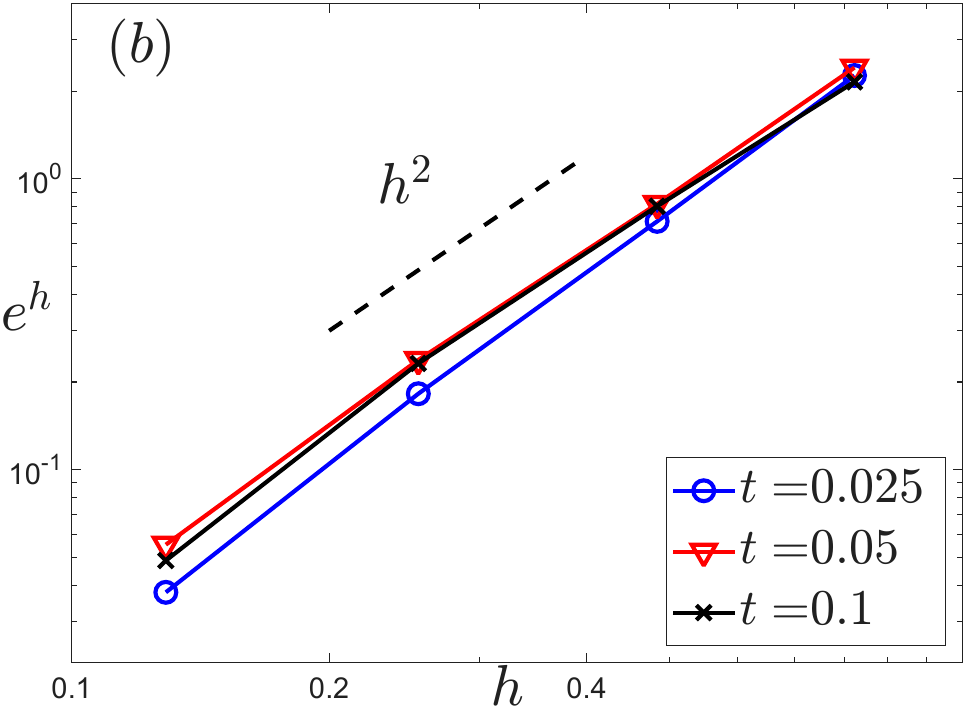}
\caption{Numerical errors $e^h$ of different initial surfaces at different time $t$: (a) unit sphere, and (b) ellipsoid with $a = 2$ and $b = 1$. Recall the time step size $\tau = 0.01h^2$.}
\label{fig: convergence}
\end{figure}
{To measure the difference between two closed surfaces, we adopt the manifold distance $\mathcal{M}(\cdot, \cdot)$ defined in \cite{zhao2021energy, bao2021structure} for closed surfaces.  Given two orientable closed surfaces $\Gamma$ and $\Gamma'$, their manifold distance is given by
\begin{equation}
    \mathcal{M}(\Gamma, \Gamma')\coloneqq |\Omega|+|\Omega'|-2|\Omega\cap\Omega'| ,
\end{equation}
where $\Omega$ and $\Omega'$ represent the regions enclosed by $\Gamma$ and $\Gamma'$, respectively.} The numerical error $e^h$ is defined as 
\begin{equation}
    e^h(t)|_{t = t_m} \coloneqq \mathcal{M}(\Gamma^m, \Gamma(t_m)).
\end{equation}

However, since analytical solutions $\Gamma(t)$ at $t = t_m$ are not available for general surfaces, we adopt the numerical results obtained from a more refined mesh size $h_e$ as our reference exact solution in convergence test, e.g. $h_{e} =0.065 $ for ellipsoid with $a = 2,b = 1$. 

\begin{figure}[t]
\centering
\includegraphics[width=0.48\textwidth]{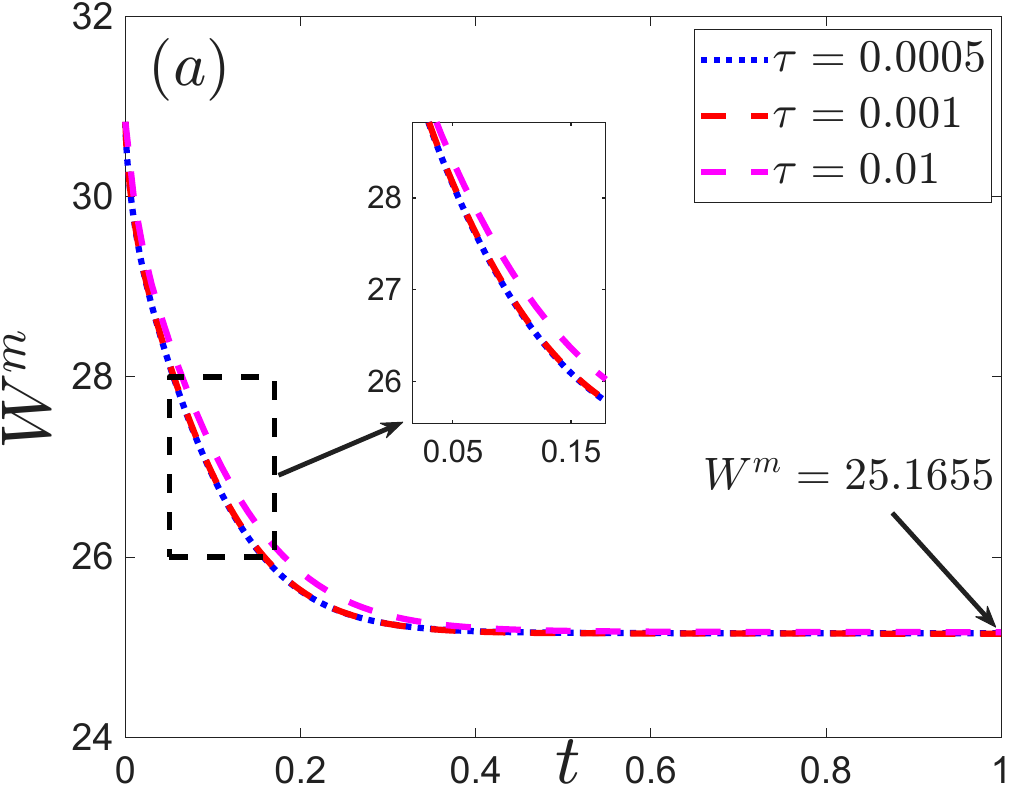}
\includegraphics[width=0.48\textwidth]{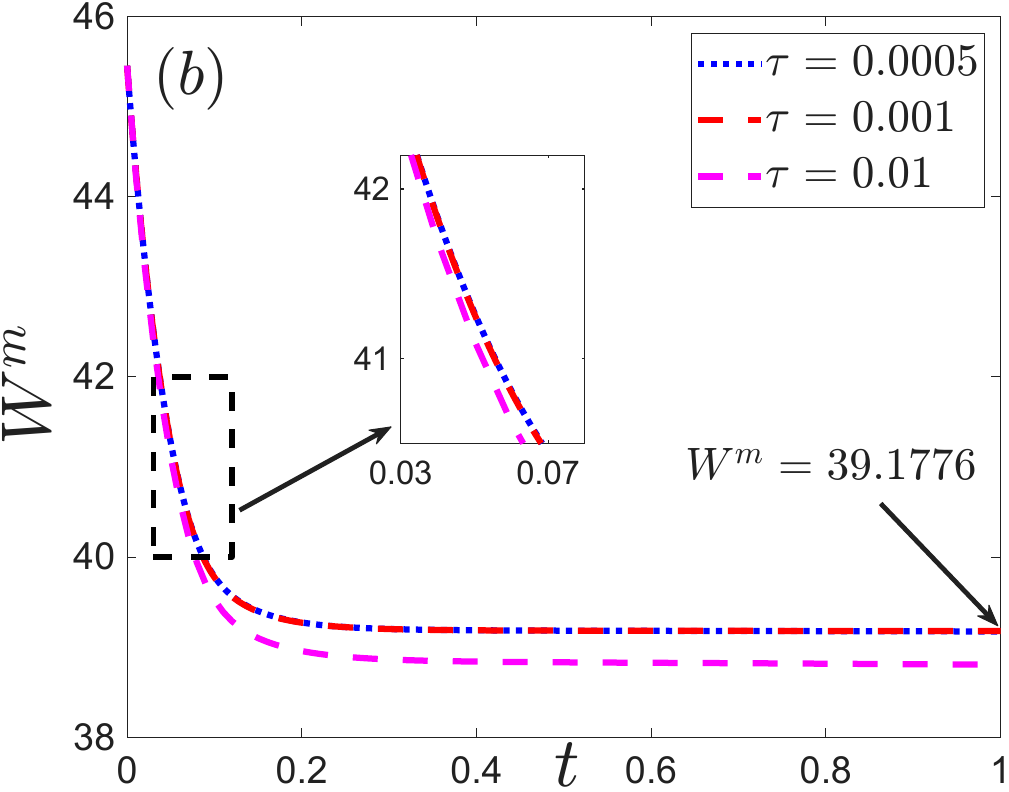}
\caption{Discretized Willmore energy $W^m$ of different initial surfaces: (a) ellipsoid with $a = 4$ and $b = 1$, $h = 0.077 $; (b)  torus with $R = \sqrt{2}$ and $r = \sqrt{2}/2$, $h = 0.065$.}
\label{fig: energy}
\end{figure}
\begin{figure}[h]
\centering
\includegraphics[width=0.9\textwidth]{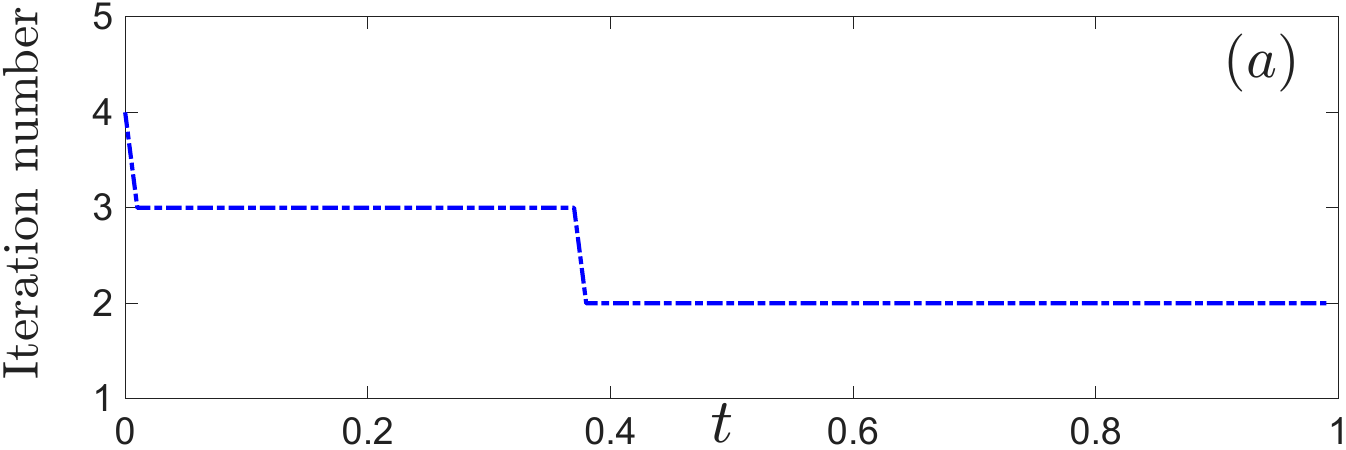}
\includegraphics[width=0.9\textwidth]{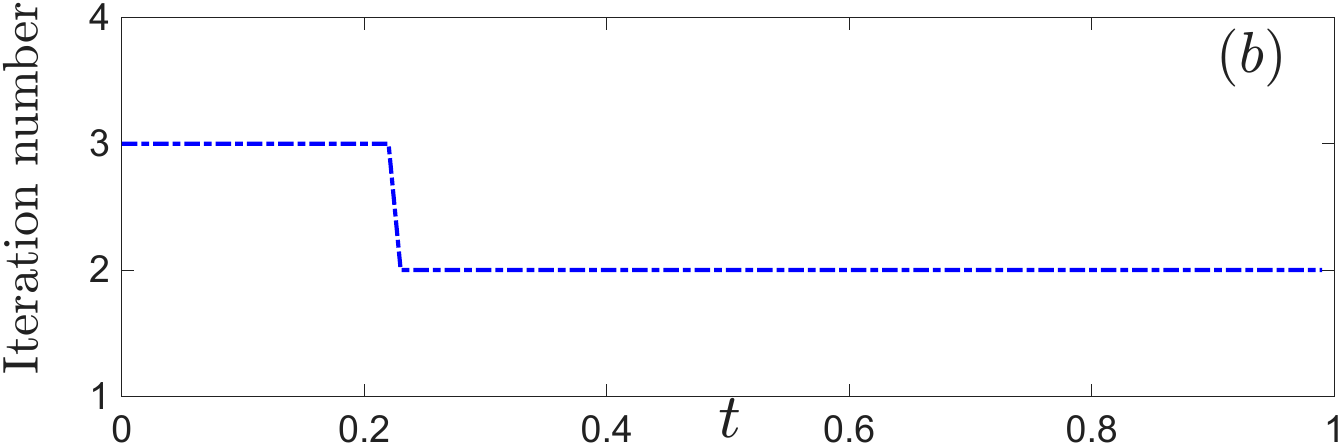}
\caption{Iteration number of \eqref{iteration} for different initial surfaces: ellipsoid with $a = 4$ and $b = 1$, $h = 0.077$ and $\tau =0.01 $; (b)  torus with $R = \sqrt{2}$ and $r = \sqrt{2}/2$, $h =0.065$ and $\tau = 0.01$.}
\label{fig: iteration}
\end{figure}

Fig. \ref{fig: convergence} illustrates the errors $e^h$ between the numerical results and the exact solutions at different times $t$. Fig. \ref{fig: energy} depicts the unconditional energy dissipation of the ES-PFEM \eqref{eqn:full wm reduced} for different time step sizes $\tau$ with fixed mesh size $h$. And Fig. \ref{fig: iteration} illustrates the number of iterations required in each time step. 

\begin{figure}[htp!]
\centering
\includegraphics[width=0.33\textwidth]{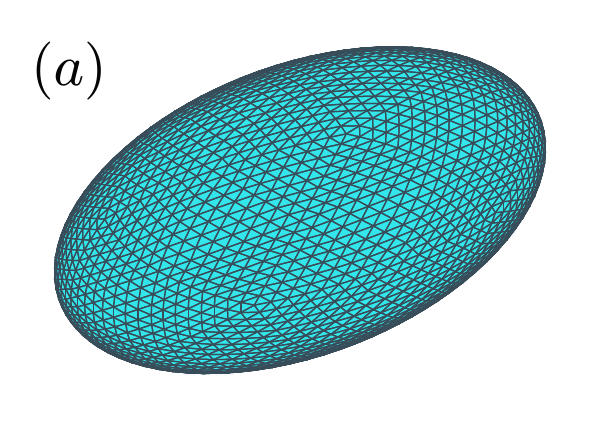}\includegraphics[width=0.33\textwidth]{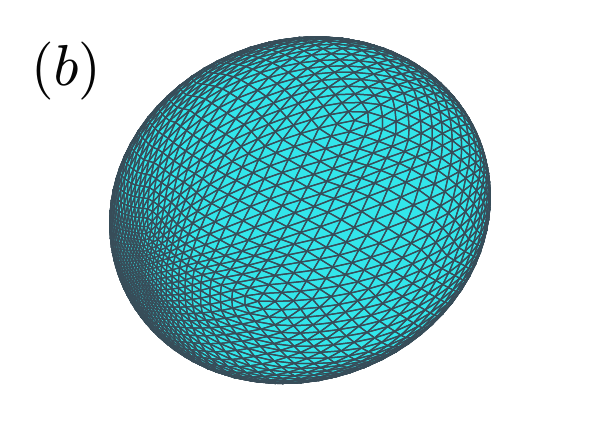}\includegraphics[width=0.33\textwidth]{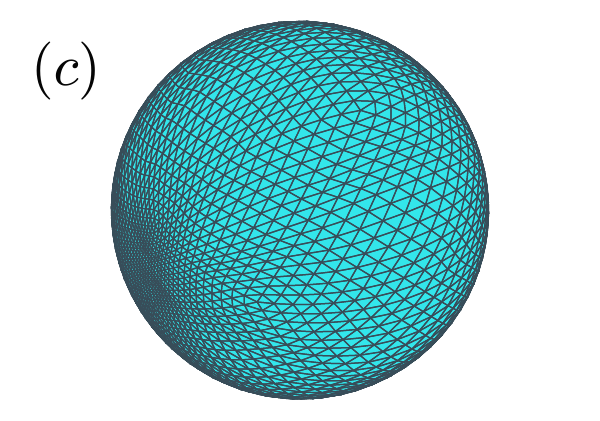}
\caption{Morphological evolution of an ellipsoid with $a = 4$ and $b = 1$ under the Willmore flow 
at (a) $t = 0$, (b) $t = 0.25$, and (c) $t = 1$. The mesh size and time step size are chosen as $h = 0.077$ and $\tau = 0.01$.}
\label{fig: evolution1}
\end{figure}

From Figs. \ref{fig: convergence}, \ref{fig: energy} and \ref{fig: iteration}, we can conclude that:
\begin{enumerate}
    \item The convergence order of the ES-PFEM \eqref{eqn:full wm reduced} achieves around $2$ robustly among our tested examples (cf. Fig. \ref{fig: convergence}). 
    \item The Willmore energy is unconditionally dissipative as proved in Theorem \ref{thm: energy2}  (cf. Fig. \ref{fig: energy}).
    \item The steady state of ellipsoid achieves an approximate minimal energy value of $8\pi\approx 25.1327$ (the result shown in Fig. \ref{fig: energy} (a) is $W^m = 25.1655$). And the steady state of torus (genus 1) yields an approximate minimal energy value of $4\pi^2\approx 39.4784$ (the result shown in Fig. \ref{fig: energy} (b) is $W^m = 39.1766$).
    \item Even though our ES-PFEM is implicit and Newton iterations are employed, it can be observed that most of the time steps are completed within two or three iterations, demonstrating the high efficiency of the Newton-Raphson iterative method for solving the problem {even with a relatively large time step $\tau =0.01$} (cf. Fig. \ref{fig: iteration}).
    
\end{enumerate}

And finally, we can investigate the morphological evolution of closed surfaces. For closed surfaces with genus $0$, we find that the surfaces tend to evolve towards spherical configurations. For the evolution of torus (genus 1), we found that it converge to the Clifford torus. The morphological evolutions of ellipsoid and torus are given in Fig. \ref{fig: evolution1} and Fig. \ref{fig: evolution2}. {
To further assess the effectiveness and robustness of the proposed method, we perform numerical experiments on a rounded cuboid and a perturbed torus, as shown in Figs. \ref{fig: evolution3} and \ref{fig: evolution4}, respectively. These examples show that the method can successfully handle evolutions with substantial geometric deformation while maintaining reliable numerical behavior, a larger initial penalty parameter is adopted. In our implementations for Figs. \ref{fig: evolution3} and \ref{fig: evolution4}, we set $\alpha^0 = 10^6$.
}

\begin{figure}[t]
\centering
\includegraphics[width=0.33\textwidth]{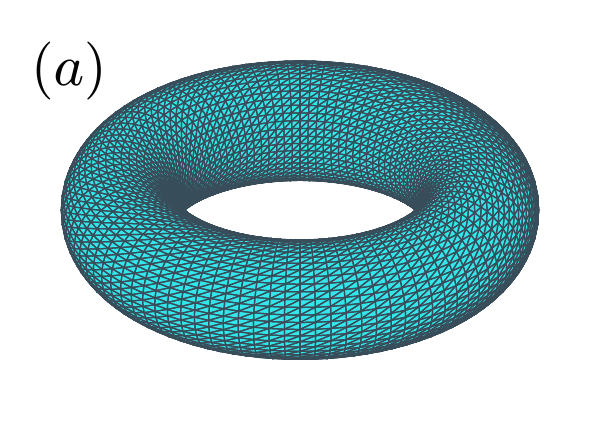}\includegraphics[width=0.33\textwidth]{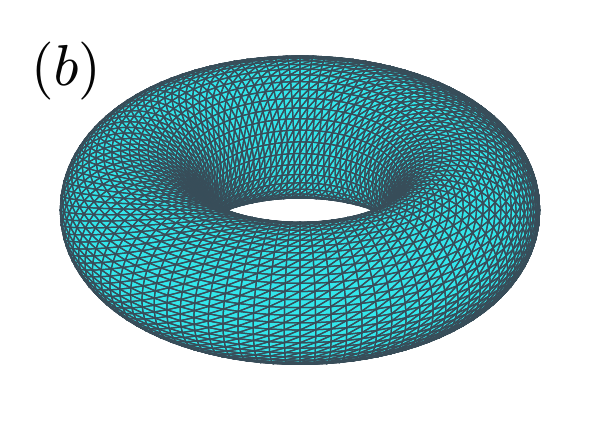}\includegraphics[width=0.33\textwidth]{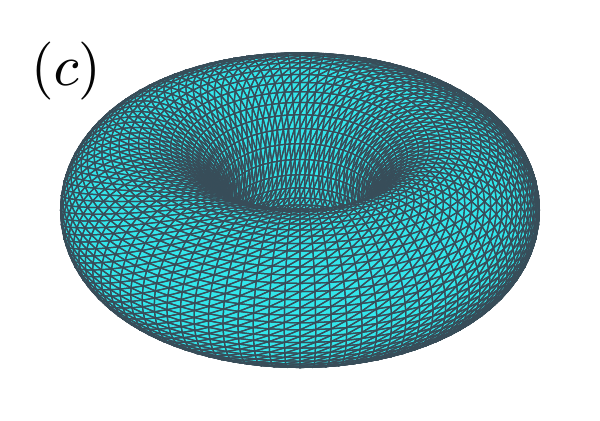}
\caption{Morphological evolution of a tours with $R = 3$ and $r = 1$ under the Willmore flow 
at (a) $t = 0$, (b) $t = 0.5$, and (c) $t = 1$. The mesh size and time step size are chosen as $h = 0.131$ and $\tau = 0.01$.}
\label{fig: evolution2}
\end{figure}

\begin{figure}[t]
\centering
\includegraphics[width=0.33\textwidth]{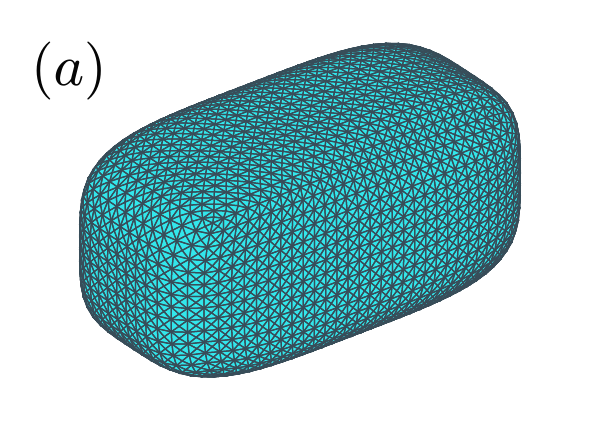}\includegraphics[width=0.33\textwidth]{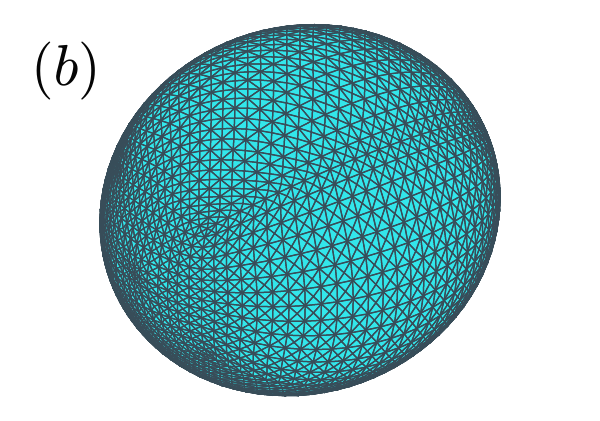}\includegraphics[width=0.33\textwidth]{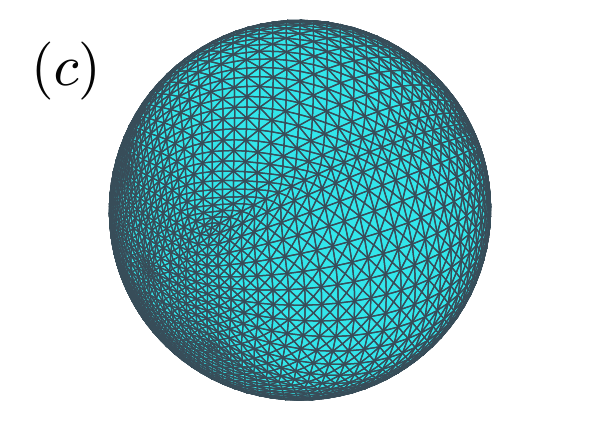}
\vspace{2em}

\includegraphics[width=0.7\textwidth]{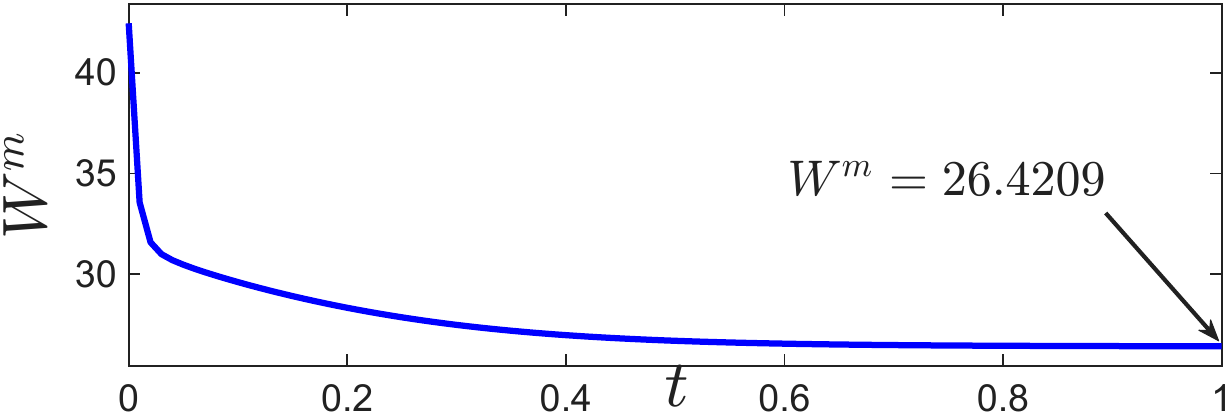}

\caption{Morphological evolution  of a rounded cuboid under the Willmore flow at  (a) $t = 0$, (b) $t = 0.5$, and (c) $t = 1$.  The bottom figure shows the corresponding discrete Willmore energy. The mesh size and time step size are chosen as $h = 0.068$ and $\tau = 0.01$ (Figures have been scaled).}
\label{fig: evolution3}
\end{figure}

\begin{figure}[t]
\centering
\includegraphics[width=0.33\textwidth]{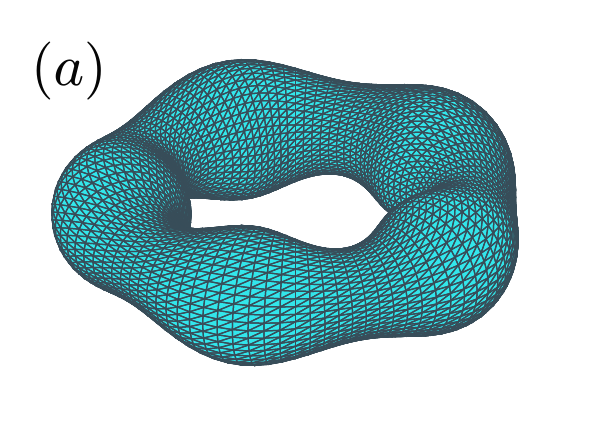}\includegraphics[width=0.33\textwidth]{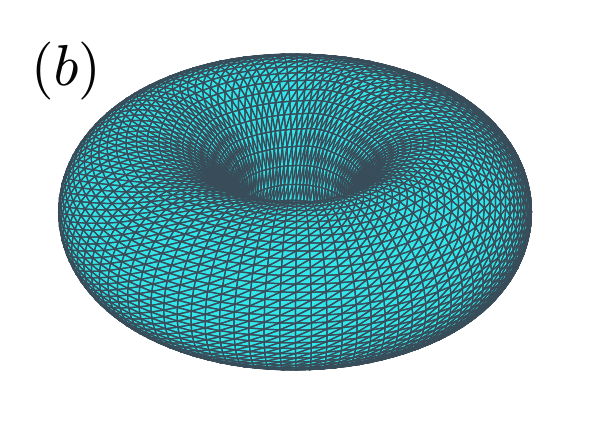}\includegraphics[width=0.33\textwidth]{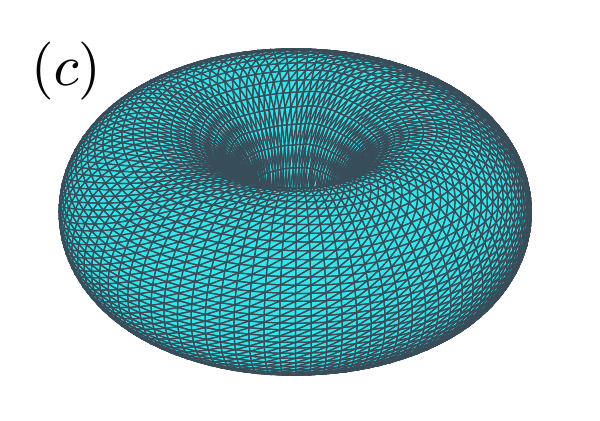}
\vspace{2em}
\includegraphics[width=0.7\textwidth]{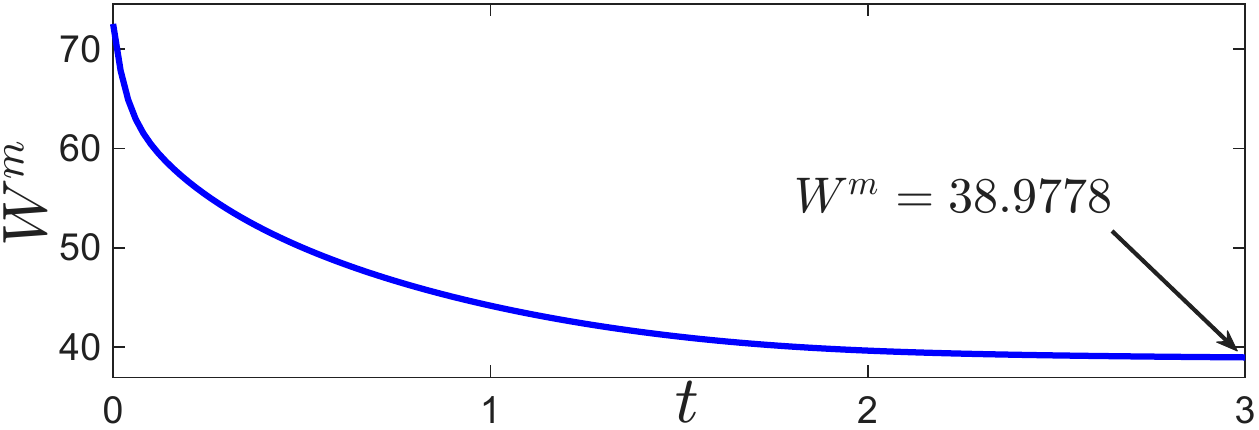}
\caption{Morphological evolution  of a perturbed torus under the Willmore flow at  (a) $t = 0$, (b) $t = 1.5$, and (c) $t = 3$.  The bottom figure shows the corresponding discrete Willmore energy. The mesh size and time step size are chosen as $h = 0.147$ and $\tau = 0.02$.}
\label{fig: evolution4}
\end{figure}

\subsection{Further numerical tests for curvature-dependent geometric gradient flows}

In this subsection, we focus on the curvature-dependent geometric gradient flows with $f(\cur) =1, \cur,\frac{1}{2}\cur^2$ and $\cur^4$.{ The improved ES-PFEMs such as \eqref{eqn:full wm reduced} for $f(\cur) = \frac{1}{2} \cur^2$ are used in our implementation.}

\begin{figure}[htp!]
    \centering
    
    \includegraphics[width=0.3\textwidth]{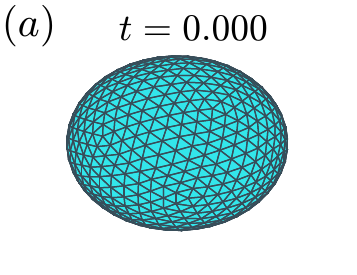}\includegraphics[width=0.3\textwidth]{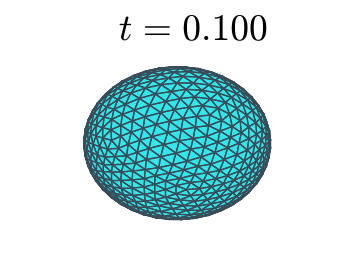}\includegraphics[width=0.3\textwidth]{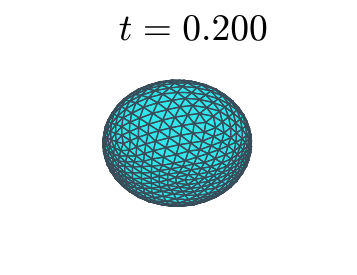}
    
    \includegraphics[width=0.3\textwidth]{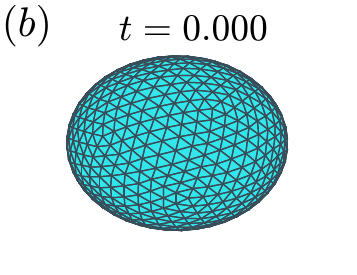}\includegraphics[width=0.3\textwidth]{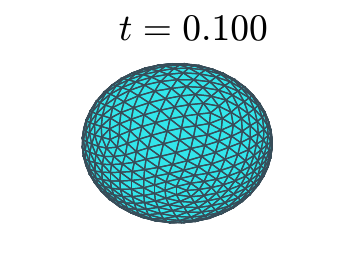}\includegraphics[width=0.3\textwidth]{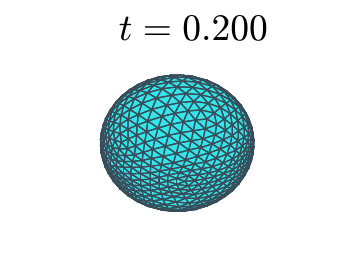}
    
    \includegraphics[width=0.3\textwidth]{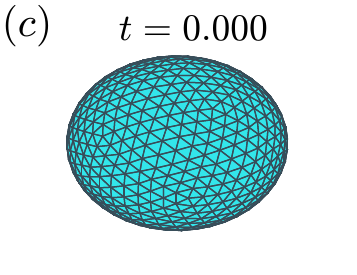}\includegraphics[width=0.3\textwidth]{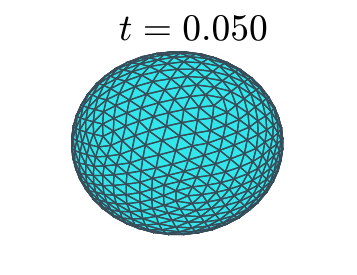}\includegraphics[width=0.3\textwidth]{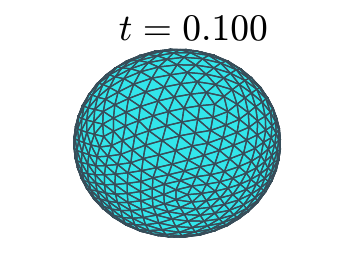}
    
    \includegraphics[width=0.3\textwidth]{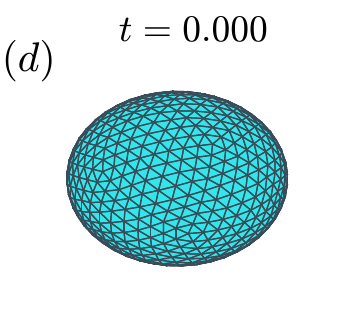}\includegraphics[width=0.3\textwidth]{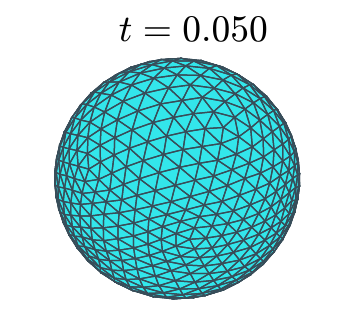}\includegraphics[width=0.3\textwidth]{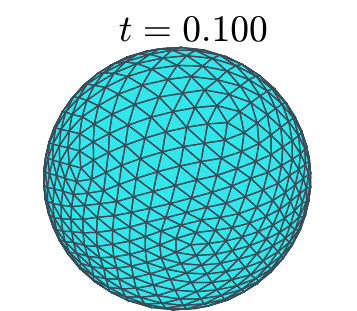}
    
  \includegraphics[width=0.9\textwidth]{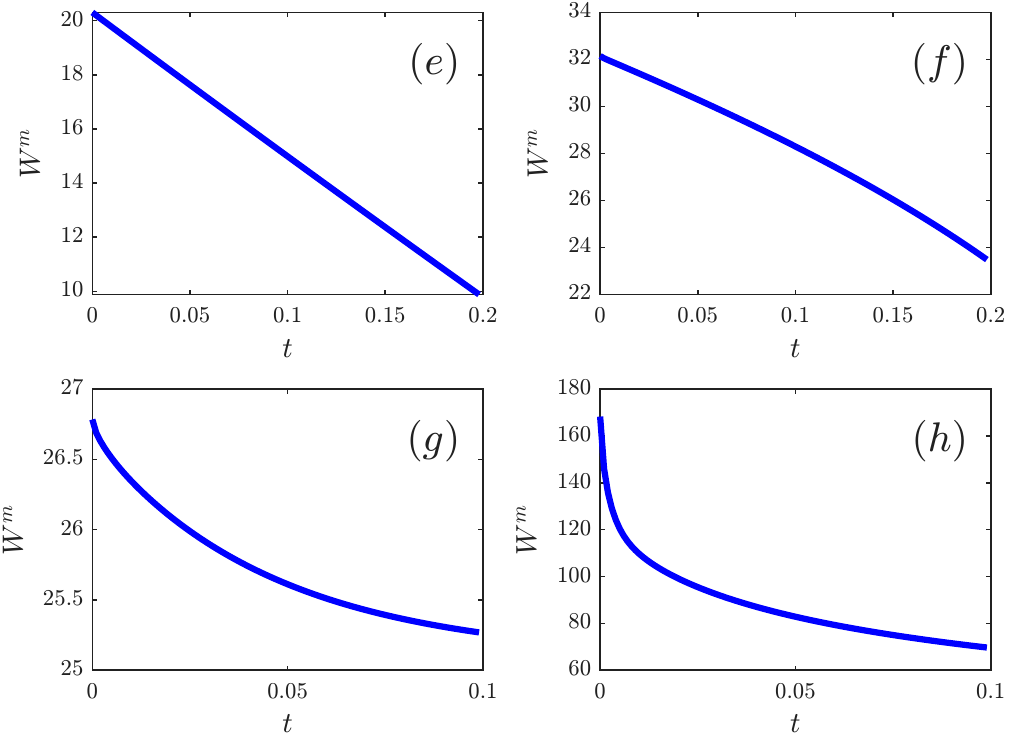}
    \caption{Morphological evolutions of ellipsoid under curvature-dependent geometric gradient flow for (a) $f(\cur) =1$ with $\tau = 0.002$; (b)$f(\cur) = \cur$ with $\tau = 0.002$; (c) $f(\cur)=\frac{1}{2}\cur^2$ with $\tau = 0.001$; and (d)$f(\cur)= \cur^4$ with $\tau = 0.001$. The corresponding discretized energy evolutions are shown in (e),(f),(g) and (h), respectively.} 
    \label{fig: curvature-dependent flow for ellipsoid}
\end{figure}
    
\begin{figure}[htp!]
    \centering
    \includegraphics[width=0.3\textwidth]{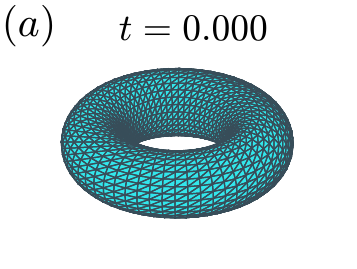}\includegraphics[width=0.3\textwidth]{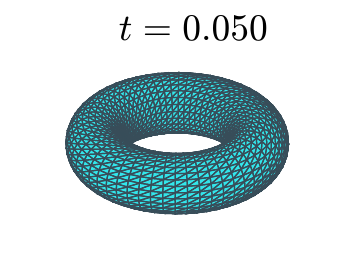}\includegraphics[width=0.3\textwidth]{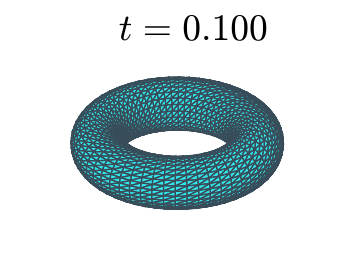}
    
    \includegraphics[width=0.3\textwidth]{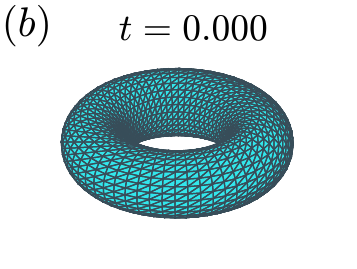}\includegraphics[width=0.3\textwidth]{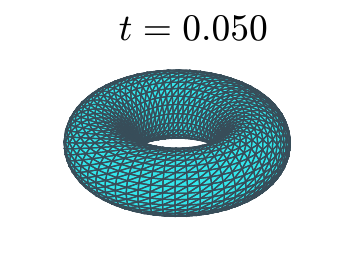}\includegraphics[width=0.3\textwidth]{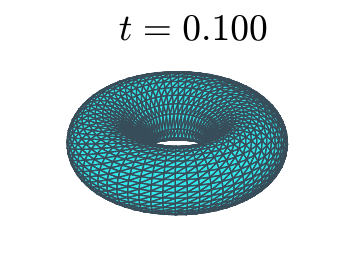}
    
    \includegraphics[width=0.3\textwidth]{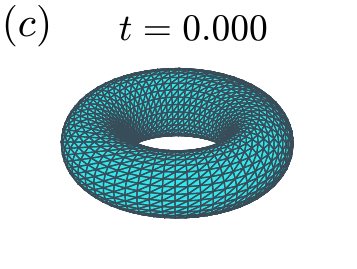}\includegraphics[width=0.3\textwidth]{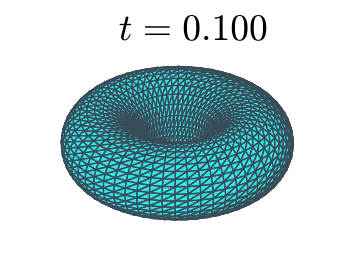}\includegraphics[width=0.3\textwidth]{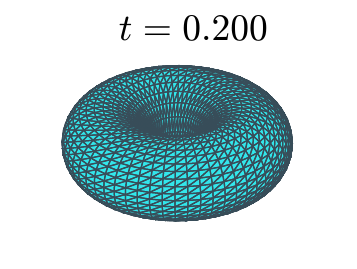}
    
    \includegraphics[width=0.3\textwidth]{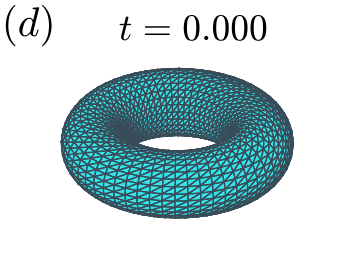}\includegraphics[width=0.3\textwidth]{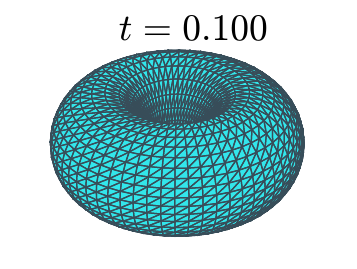}\includegraphics[width=0.3\textwidth]{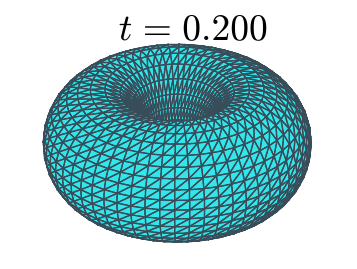}
    
    \includegraphics[width=0.9\textwidth]{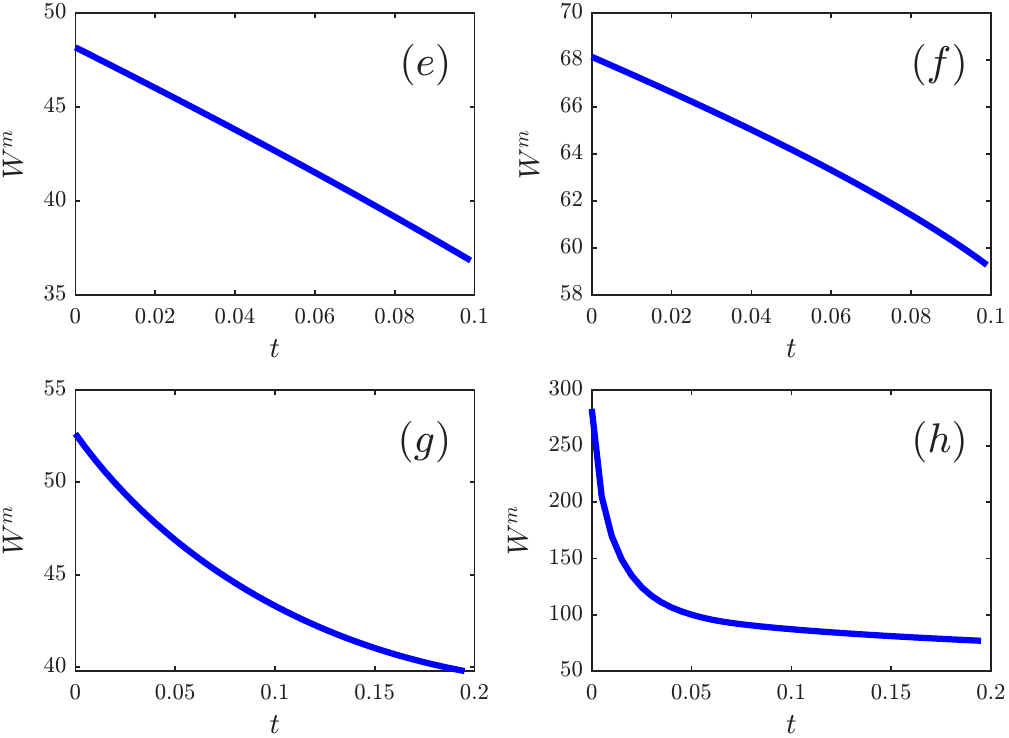}
    \caption{Morphological evolutions of torus under curvature-dependent geometric gradient flow for (a) $f(\cur) =1$ with $\tau = 0.001$; (b)$f(\cur) = \cur$ with $\tau = 0.001$; (c) $f(\cur)=\frac{1}{2}\cur^2$ with $\tau = 0.005$; and (d)$f(\cur)= \cur^4$ with $\tau = 0.005$. The corresponding discretized energy evolutions are shown in (e),(f),(g) and (h), respectively.} 
    \label{fig: curvature-dependent flow for torus}
\end{figure}

The numerical tests for morphological evolutions and the corresponding discretized energy are performed for different $f(\cur)$. In these numerical tests, we use an ellipsoid with $a = b
= 2$ and a torus with $R = \sqrt{3}, r = \sqrt{2}/2$ as initial surfaces, with mesh sizes of $h = 0.151$ and $h = 0.147$, respectively.  In particular, we set $\alpha^0 = 10^6$ in the case of torus simulation.

As illustrated in Figs. \ref{fig: curvature-dependent flow for ellipsoid}  and  \ref{fig: curvature-dependent flow for torus}, we observe that the energy is strictly dissipative. The morphological evolution shows that for $f(\cur)=1$ and $f(\cur)=\cur$, the surfaces approximately undergo contraction, whereas for $f(\cur) = \cur^4$, the surfaces exhibit expansion. We use fixed time points to compare different evolutionary processes. For $f(\cur) = \frac{\cur^2}{2}$, Figs. \ref{fig: curvature-dependent flow for ellipsoid} (g) and \ref{fig: curvature-dependent flow for torus} (g) show that the two surfaces have not achieved stability at the predetermined moment under Willmore flow. {Although the torus considered here is not a convex surface, and our ES-PFEM for Gauss curvature flow cannot theoretically guaranty energy stability, we still computed the torus for the sake of comparison. Nevertheless, the energy dissipation in Fig.~\ref{fig: curvature-dependent flow for torus}(f) remains satisfactory.}

\begin{figure}[t]
\centering
\includegraphics[width=0.33\textwidth]{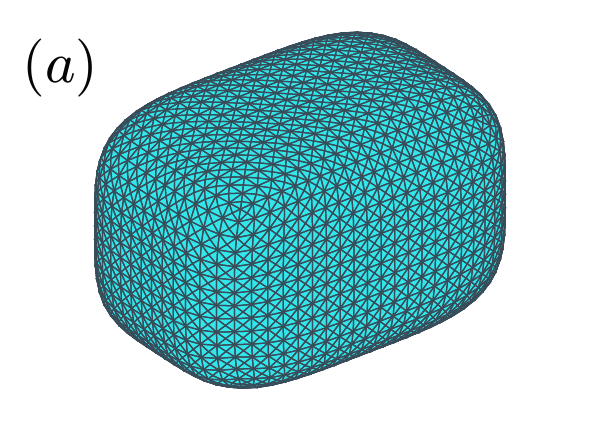}\includegraphics[width=0.33\textwidth]{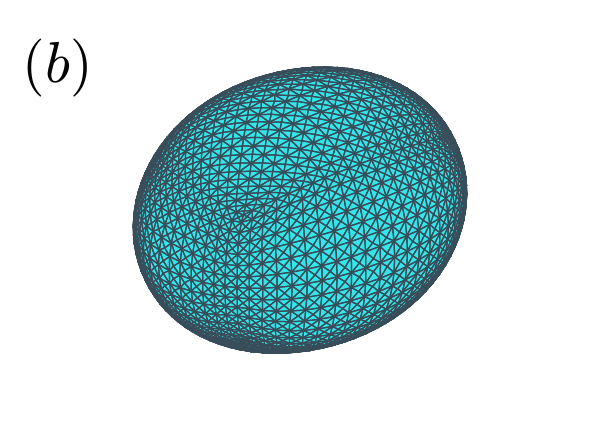}\includegraphics[width=0.33\textwidth]{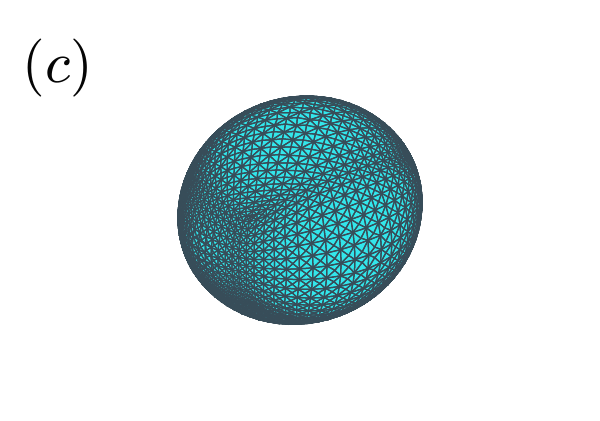}
\vspace{2em}
\includegraphics[width=0.45\textwidth]{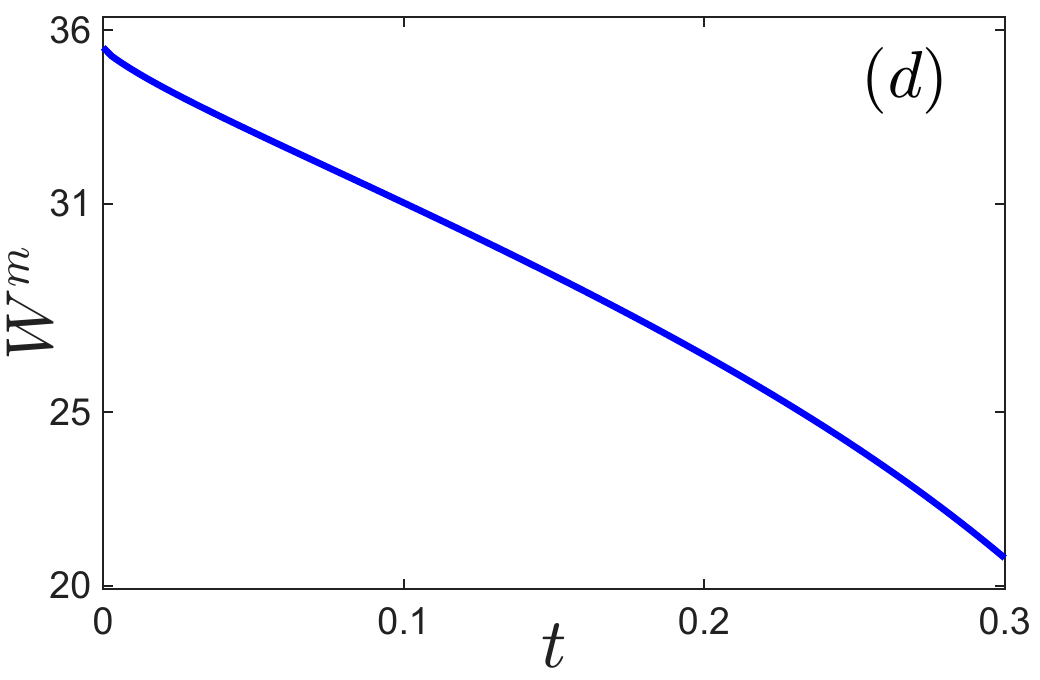}
\includegraphics[width=0.45\textwidth]{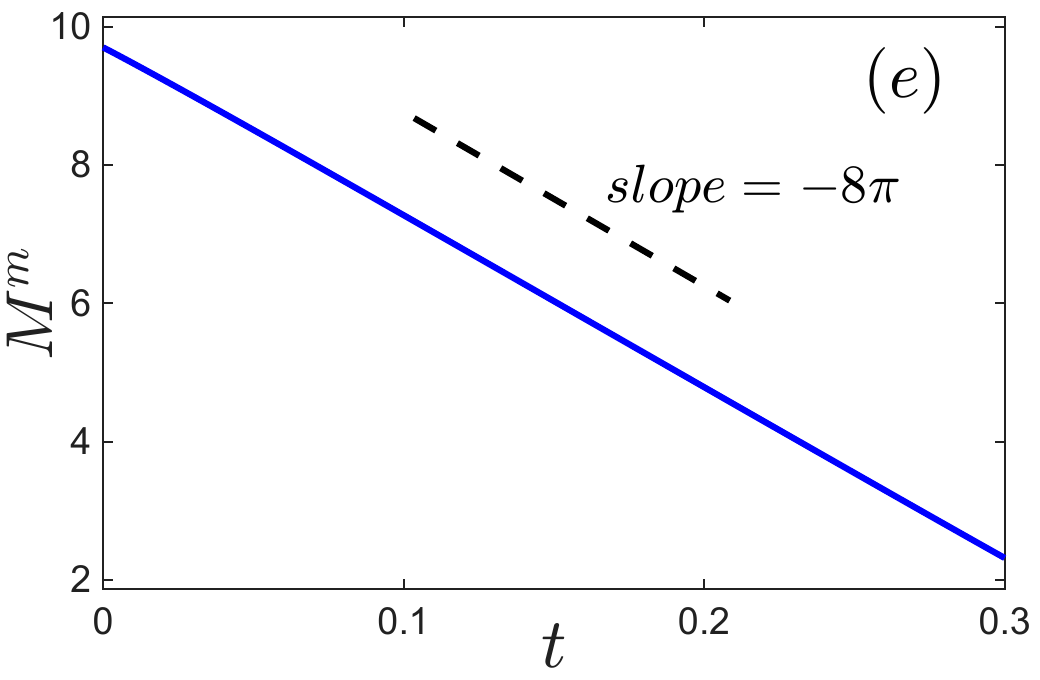}
\caption{Morphological evolution  of a rounded cuboid under the Gauss curvature flow at  (a) $t = 0$, (b) $t = 0.15$, and (c) $t = 0.3$.  The bottom figures show the corresponding  (d) discrete energy $W^m$ and (e) enclosed volume $M^m$. The mesh size and time step size are chosen as $h = 0.066$ and $\tau = 0.003$.}
\label{fig: gauss cuboid evolution}
\end{figure}

\begin{figure}[t]
\centering
\includegraphics[width=0.33\textwidth]{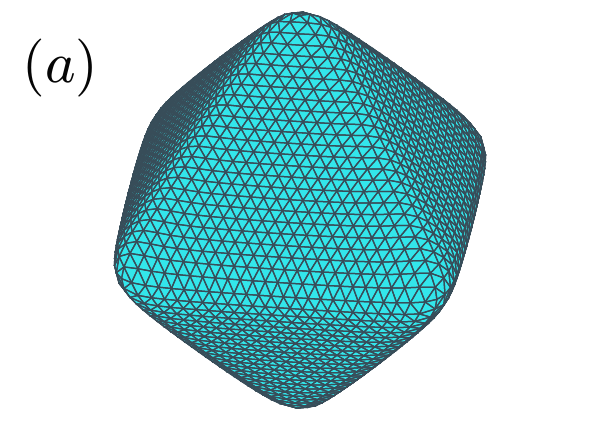}\includegraphics[width=0.33\textwidth]{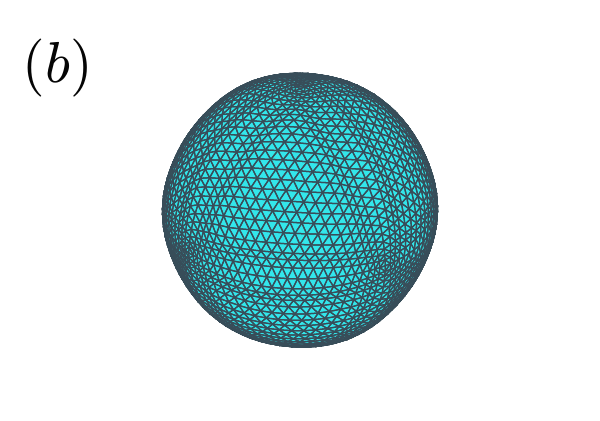}\includegraphics[width=0.33\textwidth]{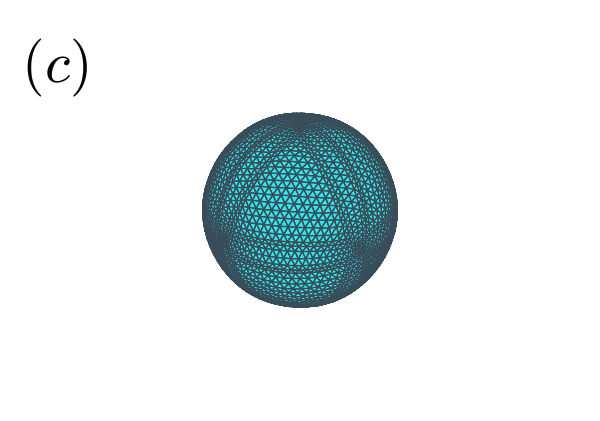}
\vspace{2em}
\includegraphics[width=0.45\textwidth]{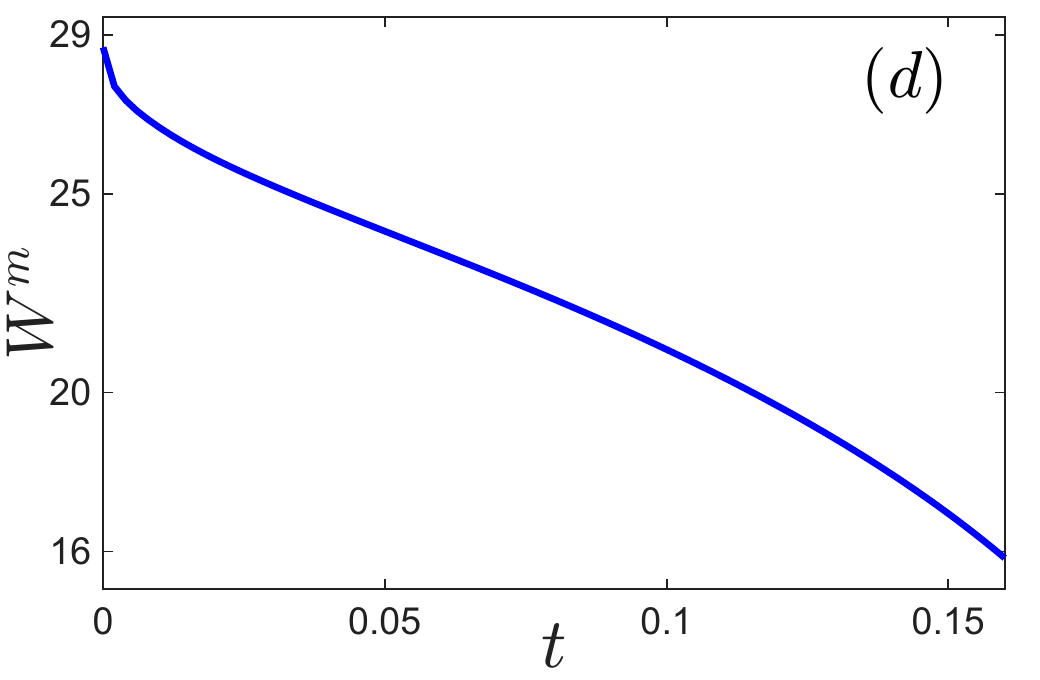}
\includegraphics[width=0.45\textwidth]{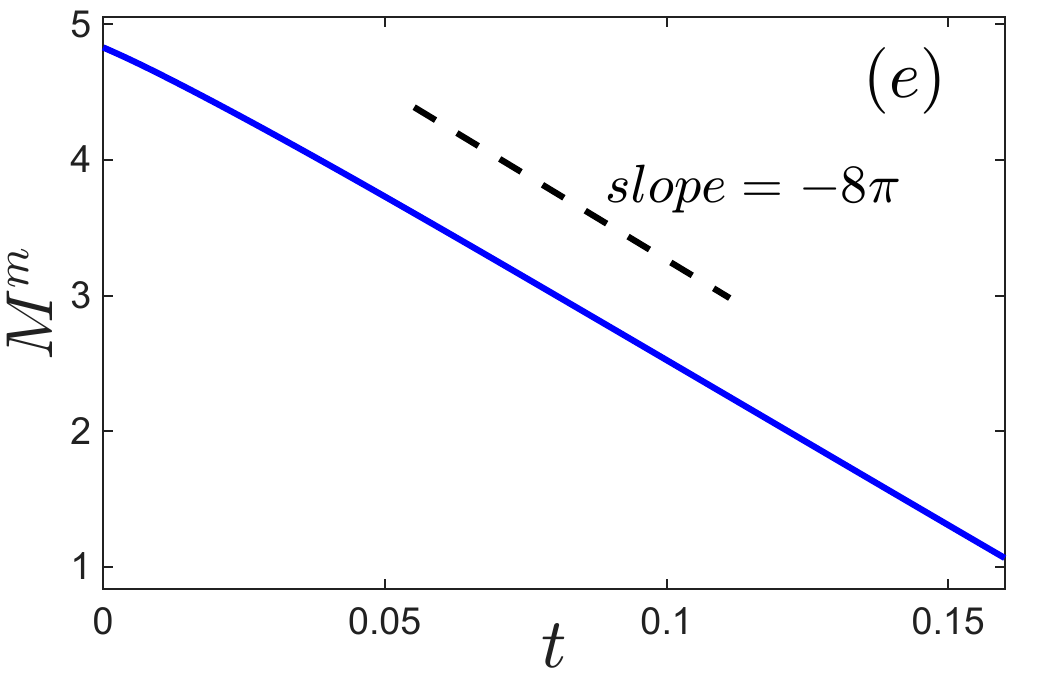}
\caption{Morphological evolution  of a rounded octahedron under the Gauss curvature flow at  (a) $t = 0$, (b) $t = 0.08$, and (c) $t = 0.16$.  The bottom figures show the corresponding (d) discrete energy $W^m$ and (e) enclosed volume $M^m$. The mesh size and time step size are chosen as $h = 0.067$ and $\tau = 0.002$.}
\label{fig: gauss octahedron evolution}
\end{figure}

\begin{figure}[t]
\centering
\includegraphics[width=0.33\textwidth]{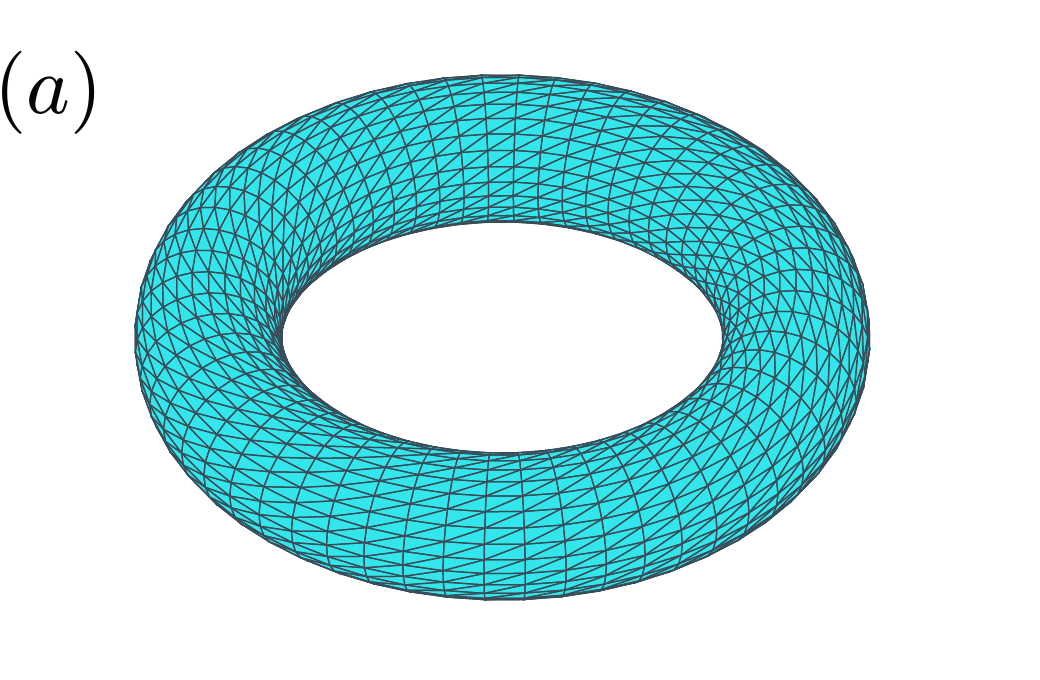}\includegraphics[width=0.33\textwidth]{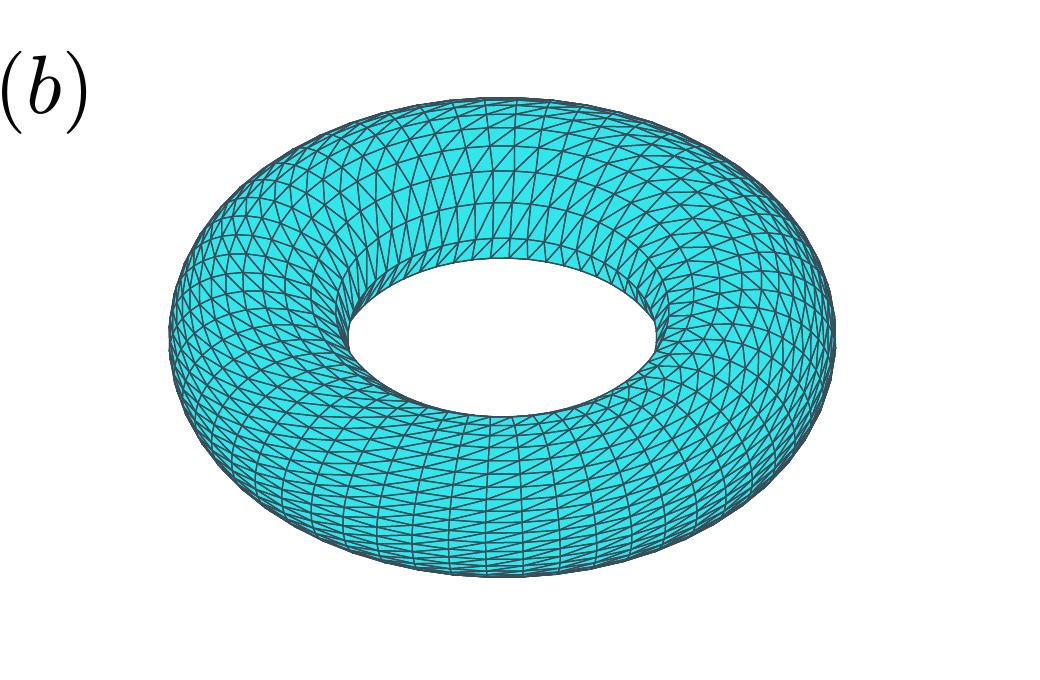}\includegraphics[width=0.33\textwidth]{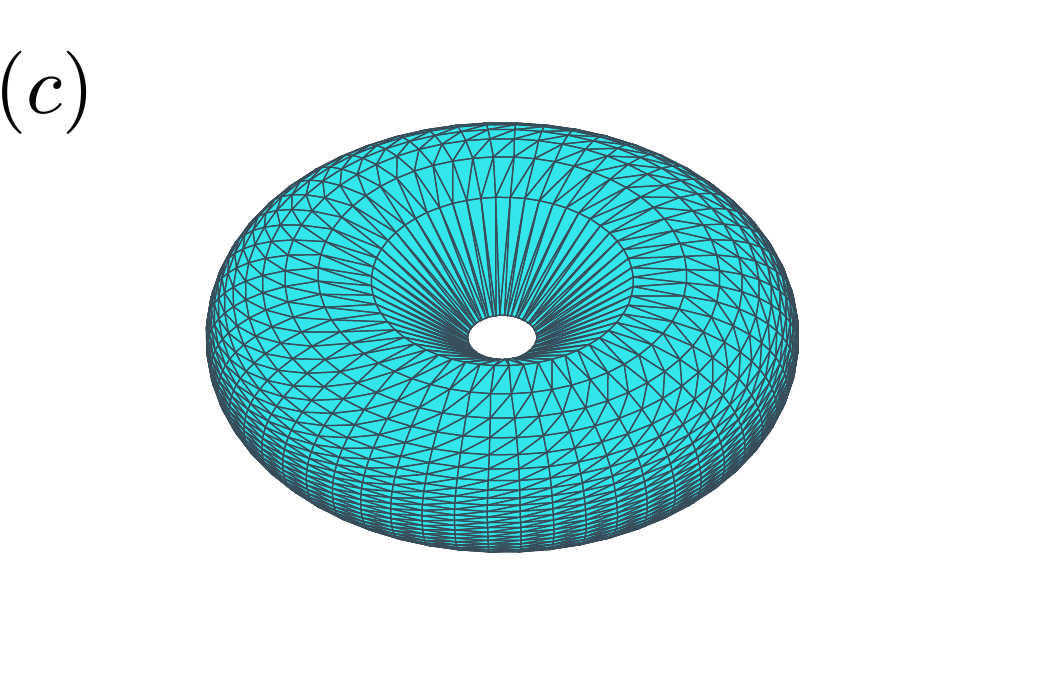}
\vspace{2em}
\includegraphics[width=0.45\textwidth]{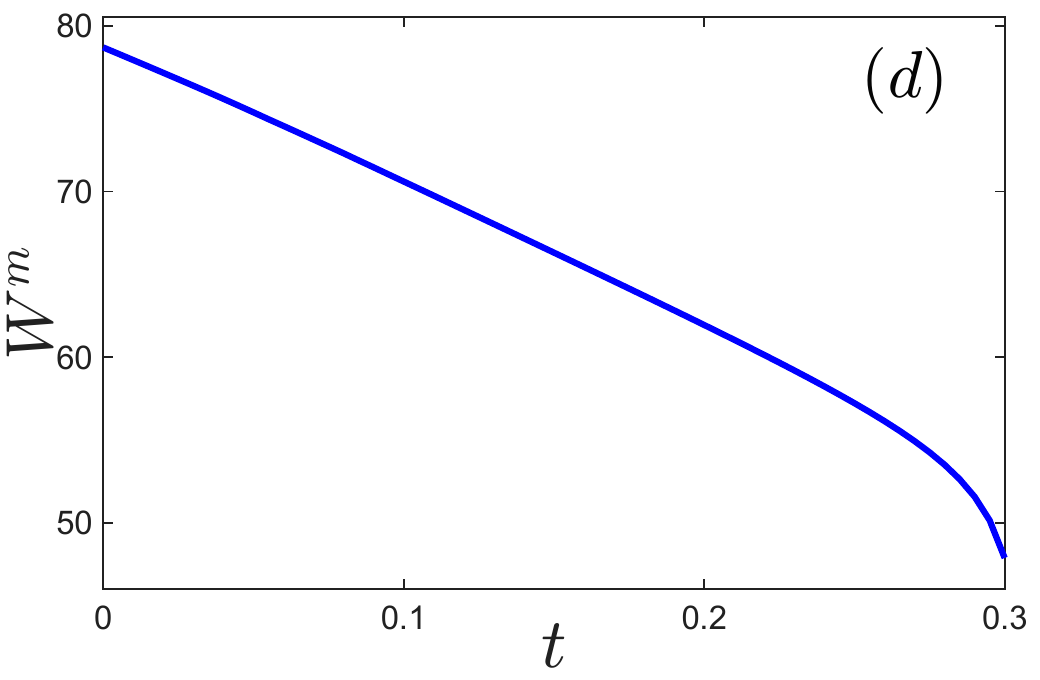}
\includegraphics[width=0.45\textwidth]{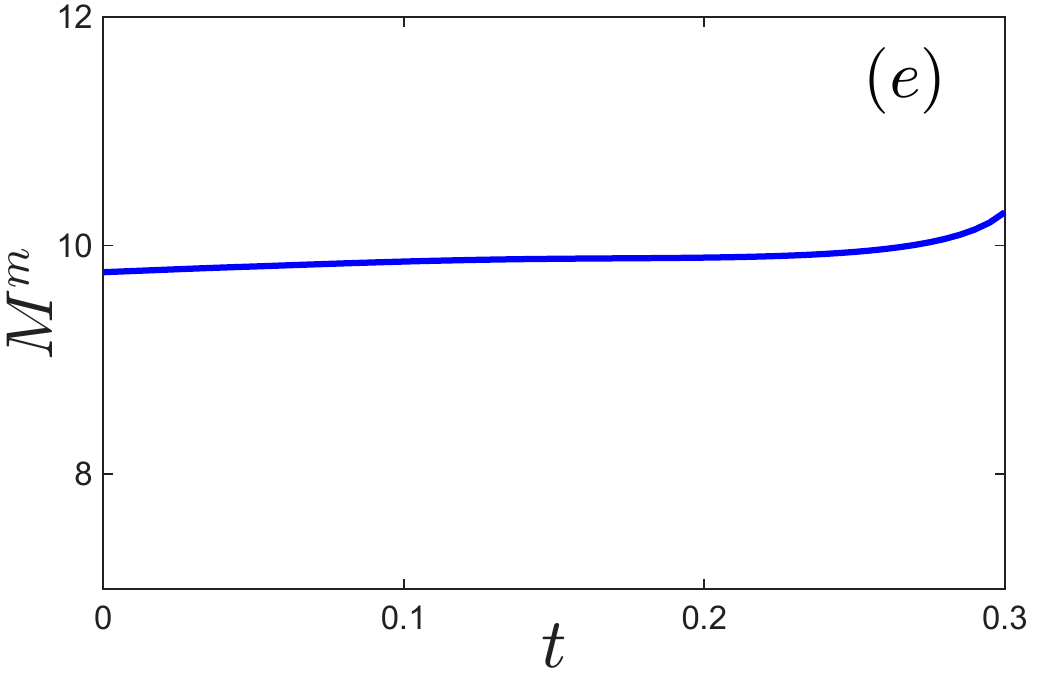}
\caption{Morphological evolution  of a torus with $R = 2$ and $r=0.5$ under the Gauss curvature flow at  (a) $t = 0$, (b) $t = 0.15$, and (c) $t = 0.3$.  The bottom figures show the corresponding (d) discrete energy $W^m$ and (e) enclosed volume $M^m$. The mesh size and time step size are chosen as $h = 0.125$ and $\tau = 0.005$.}
\label{fig: gauss torus evolution}
\end{figure}

{
The Gauss curvature flow is worthy of further numerical investigation. In Figs.~\ref{fig: gauss cuboid evolution} and~\ref{fig: gauss octahedron evolution}, we present evolutionary examples for a rounded cuboid and a rounded octahedron, respectively. At the discrete level, we denote the volume of the polygonal surface at time $t_m$ by $M^m$. According to \eqref{volume rate of gauss flow}, the enclosed volumes of these genus-zero surfaces decrease linearly in time, and the evolving shapes asymptotically approach spheres. This is corroborated by Figs.~\ref{fig: gauss cuboid evolution}(d) and~\ref{fig: gauss octahedron evolution}(d), where the slope of the volume change is approximately $-8\pi$. We further include a torus example in Fig.~\ref{fig: gauss torus evolution} as a more challenging non-convex test with sign-changing Gaussian curvature. The method captures the main qualitative deformation, while the less satisfactory volume behavior in the late stage reflects the numerical difficulty caused by the shrinking central hole and the resulting mesh deterioration.

Moreover, numerical experiments reveal that the computation of complex surfaces beyond spheres sometimes presents significant challenges due to mesh quality degeneration during the evolution, even with tangential motion control techniques.  Two approaches can address this issue: introducing a more appropriate tangential velocity \cite{duan2024new, barrett2008parametric1, garcke2025stable, hu2022evolving, kemmochi2025structure} or performing mesh redistribution \cite{bonito2010parametric}.
However, employing these techniques makes it challenging to preserve the theoretical energy dissipation of the curvature-dependent energy under full discretization. Developing more effective mesh-improvement techniques that preserve energy dissipation as in \cite{kemmochi2025structure} will be further investigated in future work.}

\section{Conclusions}
By introducing a novel identity for the normal velocity vector and an evolution equation of the mean curvature, we proposed a new consistent variational formulation for the  Willmore flow of surfaces in three dimensions (3D). On the basis of the new formulation, we obtained its full discretization by applying PFEM and proved the unconditional energy dissipation of the resulting ES-PFEM. This framework was then extended to other curvature-dependent geometric gradient flows successfully, thus providing a comprehensive insight for the design of ES-PFEM for general gradient flows.
The tangential motion control method was employed to enhance the mesh quality and robustness of the method by coupling the tangential velocity equations while maintaining favorable energy dissipation properties. Moreover, it should be mentioned  how to introduce appropriate tangential motion for long time simulations and maintaining energy stability in the dynamics remains an interesting problem. In future work, we intend to extend our ES-PFEM to Helfrich flow and other related geometric flows.

\vspace{1em}
	\noindent\textbf{Acknowledgment.} This work was partially supported by the Ministry of Education of Singapore under its AcRF Tier 1 funding A-8003584-00-00. The work of Yifei Li is funded by the Alexander von Humboldt Foundation. 

\bibliographystyle{abbrv}

 \end{document}